\documentclass{amsart}
\usepackage[utf8]{inputenc}

\usepackage{tikz} 
\usepackage{amsmath}
\usetikzlibrary{cd}
\usetikzlibrary{patterns,decorations}
\usetikzlibrary{calc} 
\usetikzlibrary{shapes.geometric}
\usepackage[utf8]{inputenc}
\usepackage{amsmath, amssymb}
\usepackage{stmaryrd}
\usepackage{mathtools}
\usepackage[mathcal]{euscript}
\usepackage{tikz}
\usepackage{dsfont}
\usepackage{float}
\usepackage{enumitem}
\usetikzlibrary{matrix}
\usepackage[all]{xy}
\definecolor{darkblue}{rgb}{0,0,0.6}
\definecolor{darkgreen}{rgb}{0.004, 0.296, 0.125}
\usepackage[ocgcolorlinks,colorlinks=true, citecolor=darkblue, filecolor=darkblue, linkcolor=darkblue, urlcolor=darkblue]{hyperref}
\usepackage[capitalize,noabbrev]{cleveref}

\newcommand{\bbF}{\mathbb F}

\newcommand{\bbZ}{\mathbb Z}

\DeclareMathOperator{\Sq}{Sq}
\DeclareMathOperator{\SO}{SO}

\DeclareMathOperator{\GL}{GL}
\DeclareMathOperator{\SL}{SL}
\DeclareMathOperator{\Qd}{Qd}

\DeclareMathOperator{\Ann}{Ann}

\DeclareMathOperator{\Res}{Res}

\DeclareMathOperator{\Rey}{\mathcal R}
\DeclareMathOperator{\chara}{char}

\numberwithin{equation}{section}
\newtheorem{theorem}[equation]{Theorem}

\newtheorem{corollary}[equation]{Corollary}
\newtheorem{lemma}[equation]{Lemma}
\newtheorem{proposition}[equation]{Proposition}

\theoremstyle{definition}

\newtheorem{definition}[equation]{Definition}
\newtheorem{question}[equation]{Question}
\newtheorem{remark}[equation]{Remark}

\pgfdeclarepatternformonly{mynewdots}{\pgfqpoint{-1pt}{-1pt}}{\pgfqpoint{1pt}{1pt}}{\pgfqpoint{1pt}{1pt}}
{%
  \pgfpathcircle{\pgfqpoint{0pt}{0pt}}{.5pt}%
  \pgfusepath{fill}%
}%

\title[Steenrod closed parameter ideals in the cohomology of $A_4$]{Steenrod closed parameter ideals in the mod-$2$ cohomology of $A_4$ and $\SO(3)$}

\author{Henrik Rüping}
\author{Marc Stephan}
\author{Erg{\" u}n Yal{\c c}{\i }n}

\address{Continentale Krankenversicherung a.G., Ruhrallee 94, 44139 Dortmund, Germany}
\email{henrikrueping@googlemail.com}
\address{Faculty of Mathematics, Bielefeld University, PO Box 100 131, 33501 Bielefeld, Germany}
\email{marc.stephan@math.uni-bielefeld.de}

\address{Department of Mathematics, Bilkent University, 06800 
Bilkent, Ankara, Turkey}

\email{yalcine@fen.bilkent.edu.tr}

\date{\today}



\begin{document}

\begin{abstract}
In this paper, we classify the parameter ideals in $H^*(BA_4;\bbF_2)$ and in the Dickson algebra $H^*(B\SO(3);\bbF_2)$ that are closed under Steenrod operations. Consequently, we obtain restrictions on the dimensions $n,m$ for which $A_4$ (and $\SO(3)$) can act freely on $S^n\times S^m$.
\end{abstract}
\maketitle

\section{Introduction}
It is a classical problem to classify all finite groups which can act freely on a sphere $S^n$. This problem is still open for some dimensions $n$. If we allow the dimension of the sphere to vary, then the problem is solved with the works of Smith, Swan, Milnor, and Madsen-Thomas-Wall (see \cite{hambleton2015} for a survey). It is proved that a finite group $G$ acts freely on $S^n$ for some $n\geq 1$ if and only if all the subgroups of $G$ with order $p^2$ and $2p$ are cyclic for all primes $p$. The first condition that every subgroup of order $p^2$ is cyclic is equivalent to the condition that the cohomology of the group $G$ in $\bbF_p$-coefficients is periodic (see \cite[Definition~6.1]{adem2013cohomology} for a definition). The dimension of the sphere $S^n$ depends on the periodicity of the group cohomology $H^*(BG; \bbZ)$, however there are also other obstructions which can make the dimension bigger than the periodicity. It requires delicate computations involving number theory to find the exact dimensions of spheres on which a finite group with periodic cohomology can act freely (see \cite{hambleton2015}).

For every prime $p$, the \emph{$p$-rank} of a finite group $G$ is defined to be the maximum integer $s$ such that $(\bbZ/p)^s$ injects into $G$, and the \emph{rank} of $G$ is the maximum of its $p$-ranks over all primes $p$ dividing the order of $G$. It is conjectured by Benson and Carlson \cite{bensoncarlson1987} that any rank $r$ group acts freely on a finite CW-complex homotopy equivalent to a product of $r$ spheres $S^{n_1}\times \cdots \times S^{n_r}$ for some $n_1, \ldots, n_r \geq 1$.
Although this conjecture is still open in full generality, some results are known for the existence of such actions. In particular, any rank $2$ finite group which does not involve $\Qd(p)=(\mathbb Z /p)^2 \rtimes \SL_2(p)$ for any odd prime $p$ admits such actions (see \cite{jackson2007qd} and \cite{adem2001periodic}).

If a finite group $G$ is known to act freely on a product of two spheres $S^n\times S^m$, it is interesting to ask which dimensions $n$ and $m$ are possible for such actions. In particular, one can ask if every rank $2$ group acts freely on $S^n \times S^n$ for some $n$. This question was answered negatively by Oliver in \cite{Oliver} where he proved that the alternating group on $4$ elements $A_4$ can not act freely on a finite CW-complex $X$ with $H^* (X; \bbZ )\cong H^*(S^n \times S^n; \bbZ)$. Oliver proved this result by considering ideals in $H^* (BA_4; \bbF_2)$ which are closed under Steenrod operations. The cohomology ring $$H^* (BA_4; \bbF_2)\cong H^* (B(\bbZ/2 \times \bbZ/2); \bbF_2 )^{C_3}$$ is isomorphic to the algebra $\bbF_2[u,v,w]/\langle u^3+v^2+vw+w^2\rangle,$ where $u$ has degree $2$ and $v,w$ have degree $3$ (see \cref{thm:AdemMilgram}). Oliver proved the following:

\begin{theorem}[{\cite[Lemma~1]{Oliver}}]\label{thm:oliver} 
Let $I\subset H^*(BA_4;\bbF_2)$ be a nonzero ideal generated by homogeneous elements of the same degree $i$. If $I$ is closed under the Steenrod operations, then $I=\langle v^k\rangle$ for $k=i/3$.
\end{theorem}

If there were a free $A_4$-action on a finite CW-complex homotopy equivalent to a product of spheres $S^n \times S^n$, then the kernel of the map $H^*(BA_4;\bbF_2)\rightarrow H^*(X/A_4;\bbF_2)$ would be such an ideal $I$ for $i=n+1$ and $H^*(BA_4; \bbF_2)/I$ would be finite over $\bbF_2$. But for $k>0$, the ring $H^*(BA_4; \bbF_2)/\langle v^k \rangle$ is not finite. This argument holds also for finite CW-complexes with mod-$2$ cohomology of $S^n\times S^n$ by \cref{thm:nofreeactionequidim}.

In this paper we look at the question of Blaszczyk from \cite[Section~4]{blaszczyk2013free}: which dimensions $n$ and $m$ are possible if $G=A_4$ acts freely on a finite CW-complex $X\simeq S^n \times S^m$ for $n\neq m$? To avoid trivial situations, we assume $n, m\geq 1$. 

We show in \cref{pro:ForA_4F_2} that if such an action exists then the kernel of the homomorphism $H^*(BA_4;\bbF_2)\rightarrow H^*(X/A_4;\bbF_2)$ is a Steenrod closed ideal $I$ in $H^*(BA_4;\bbF_2)$ with finite quotient, generated by the two $k$-invariants of the action. These $k$-invariants have degrees $m+1$ and $n+1$.

Our main result classifies the Steenrod closed parameter ideals in $H^*(BA_4;\bbF_2)$, i.e., the Steenrod closed ideals $I$ generated by two homogeneous elements such that $H^*(BA_4; \bbF_2)/I$ is finite. These can be grouped into three cases. 

\begin{theorem}[\cref{thm:classification_steenrodclosedparameterideals}, \cref{pro:uniquenessindegrees}]\label{thm:Intro-classification_steenrodclosedparameterideals}  The set of Steenrod closed parameter ideals in $H^*(BA_4;\bbF_2)$ is the disjoint union of 
\begin{enumerate}
\item\label{thm:Intro-fibered} the \emph{fibered ideals} $\langle v^k,u^l\rangle$ with $l\ge 1$ and $1\le k\le 2^t$, where $2^t$ is the largest power of $2$ dividing $l$;
\item\label{item:Intro-twisted} the \emph{twisted ideals} $\langle x_n,\Sq^1(x_n)\rangle$ for $n\ge 2$, where $x_n$ is recursively defined as 
$x_1=u$ and $x_{n+1} = ux_n^2+\Sq^1(x_n)^2$;
\item\label{item:Intro-mixed} and \emph{the mixed ideals} $\langle v^ix_n^{2^m},x_{n+1}^{2^{m-1}}\rangle$, where $m\ge 1$, and either
$n=1$ and $1\le i <2^{m-1}$, or $n \ge 2$ and $0\le i<2^{m-1}$.
\end{enumerate}
Moreover, for a given pair of natural numbers  there is at most one Steenrod closed parameter ideal $I$ with parameters of these degrees.
\end{theorem}
To prove \cref{thm:Intro-classification_steenrodclosedparameterideals}, we first consider the twisted ideals, i.e., the Steenrod closed parameter ideals of the form $\langle X, \Sq^1(X)\rangle$. We prove in 
\cref{thm:classification_twisted_pairs} that all such $X$ can be computed recursively as in \cref{thm:Intro-classification_steenrodclosedparameterideals}\eqref{item:Intro-twisted}. In order to show that every twisted ideal is of this form we use Kameko maps (see \cref{def:kameko}) similarly to Oliver's proof of \cref{thm:oliver}.

The nontwisted ideals can be constructed from ideals with parameters of smaller degrees as explained in \cref{sec:obsNonTwisted}.

Next, there is the fibered case, where $I$ has a system of parameters, such that one of the parameters generates a Steenrod closed ideal. The classification of the fibered ideals is in \cref{thm:classification_fibered} extending work of Meyer and Smith \cite{meyersmith2003}.

Finally, we consider the mixed case. These are all other Steenrod closed parameter ideals. They can be obtained by applying certain operations to the ideals in the twisted case; see \cref{sect:MixedCase}.

Similar to Oliver's result, we obtain the following obstruction for the existence of free actions on a product of spheres:

\begin{theorem}[\cref{thm:ObstructionskInvariants}, \cref{thm:ObstructionskInvariants_integral}]\label{cor:Intro-Topological}  
Let $R=\bbZ$ or $\bbF_2$. Suppose that $G=A_4$ acts freely on a finite CW-complex $X$ such that $H^*(X; R) \cong H^* (S^n\times S^m; R)$ with $1\leq n < m$. Then 
the classifying map induces a surjection $H^* (BG; \bbF_2) \to H^* (X/G ; \bbF_2)$ whose kernel $J$ is a Steenrod closed parameter ideal generated by the $k$-invariants of the spheres. Moreover, the following holds:
\begin{enumerate}
\item If $R=\bbF_2$, then the ideal $J$ must be in one of the three types listed in \cref{thm:Intro-classification_steenrodclosedparameterideals}. \item  If $R=\bbZ$, then the ideal $J$ must be in one of the three types listed in \cref{thm:Intro-classification_steenrodclosedparameterideals} except the twisted ones listed in \eqref{item:Intro-twisted}.
\end{enumerate}
\end{theorem}
Since there is only one Steenrod closed parameter ideal with parameters in given degrees, the cohomology of the quotient space $H^*(X/G;\bbF_2)$ in \cref{cor:Intro-Topological} only depends on the degrees $n$ and $m$.

The classification of the Steenrod closed parameter ideals provides restrictions for the degrees $m+1,n+1$ of the $k$-invariants. All possible degrees for parameters are listed in \cref{cor:DimensionsFreeAction}. These provide restrictions for dimensions $m$, $n$ such that $S^n\times S^m$ admits a free $A_4$-action. Since every finite simple group of rank $2$ contains $A_4$, this gives restrictions for free actions on products of two spheres for any finite simple group of rank $2$.

As summarized in \cite{blaszczyk2013free}, the previously known cases of products of spheres on which $A_4$ can not act freely are $S^n\times S^n$ by \cite{Oliver}, $S^1\times S^n$ by \cite[Proposition~4.3]{blaszczyk2013free}, and $S^{2n}\times S^{2m}$ since the order of $A_4$ divides the Euler characteristic of the product. This eliminates only about $1/4$ of the possible dimensions. We eliminate almost all dimensions in the following sense:

\begin{theorem}[{\cref{cor:percentage}}] For $r\geq 1$, the percentage of those pairs $(n,m)$ with $n,m\leq r$ such that there exists a finite, free $A_4$-CW-complex $X\simeq S^n\times S^m$ tends to zero as $r\to \infty$.
\end{theorem}

Using fixity methods from \cite{ADU2004}, \cite{UY2010}, and the methods due to Adem-Smith introduced in \cite{adem2001periodic}, one can realize some of the fibered ideals by a finite, free $A_4$-CW-complex $X\simeq S^n\times S^m$. We do not know if all fibered ideals are realizable. We do not know either if any of the mixed ideals is realizable.

\begin{question} Are the mixed ideals given in \cref{thm:Intro-classification_steenrodclosedparameterideals} realizable by a finite, free $A_4$-CW-complex $X$ homotopy equivalent to $S^n \times S^m$? 
\end{question}

The classification given in  \cref{thm:Intro-classification_steenrodclosedparameterideals} has implications for the Dickson algebra $D(2) = H^*(B\SO(3);\bbF_2)\cong \bbF_2[a,b]^{\GL_2 (2)}\cong \bbF_2[u,v]$ which is a subalgebra of $H^*(BA_4; \bbF_2)$. Meyer and Smith classified in \cite[Theorem~V.1.1]{meyersmith2003} all Steenrod closed parameter ideals in the Dickson algebra $D(2)$ generated by powers of $u$ and $v$.
We show in \cref{cor:classification_SteenrodclosedParameterIdeal_SO3} that restriction of ideals from $H^*(BA_4; \bbF_2)$ to $H^*(B\SO(3); \bbF_2)$ induces a bijection on Steenrod closed parameter ideals. Hence \cref{thm:Intro-classification_steenrodclosedparameterideals} provides a full classification of the Steenrod closed parameter ideals in the Dickson algebra $D(2)$. Since $A_4$ is a subgroup of $\SO(3)$, the restrictions for free $A_4$-actions on products of spheres also apply to $\SO(3)$.

\subsection*{Convention} By an action of a finite group $G$ on a CW-complex $X$, we mean that the $G$-space $X$ admits the structure of a $G$-CW-complex.

\subsection*{Acknowledgments} It is our pleasure to thank Dave Benson, Bob Oliver, and Kate Ponto for helpful discussions and providing references. 
We thank the referee for careful reading of the paper and for suggested improvements.
 
\section{Recollection of invariant theory}\label{sect:invariantTheory}

We are interested in the mod-$2$ cohomology rings of the groups  $\SO(3)$, 
$A_4=(\bbZ/2\times \bbZ/2)\rtimes C_3$, and $\bbZ/2 \times \bbZ/2$. 
We denote the mod-$2$ cohomology ring of the group $G$ by $H^*(BG)$. 
The group cohomology of $\bbZ/2\times \bbZ/2$ in $\bbF_2$-coefficients 
is the polynomial ring $\bbF_2[a,b]$ 
and the other two are the invariant rings  
\begin{align*}
    H^*(B\SO(3)) &\cong \bbF_2[a,b]^{\GL_2(2)}, \\
    H^*(BA_4) & \cong \bbF_2[a,b]^{C_3},
\end{align*}
where a fixed generator of the cyclic group $C_3$ of order $3$ acts on the vector space $\bbF_2a \oplus \bbF_2 b$ of homogeneous polynomials of degree $1$ by $a\mapsto b$, $b\mapsto a+b$. The mod-$2$ cohomology of the group $A_4$ can be described as an algebra as follows:  
 
\begin{theorem}[{\cite[Chapter~III, Theorem~1.3]{adem2013cohomology}}]\label{thm:AdemMilgram} We have
\[H^*(BA_4)\cong \bbF_2 [u,v,w]/\langle u^3+v^2+vw+w^2 \rangle \]
where $\deg (u)=2$ and $\deg(v)=\deg(w)=3$. Under the isomorphisms
$$H^* (BA_4) \cong H^*(B(\bbZ/2)^2)^{C_3} \cong \bbF_2 [a, b]^{C_3}$$
the generators correspond to
\begin{align*}
    u&= a^2+ab+b^2\\
    v&= a^2b+ab^2=ab(a+b)\\
    w&= a^3+a^2b+b^3.
\end{align*}
\end{theorem}

The invariant ring $\bbF_2[a,b]^{\GL_2(2)}$ is the \emph{Dickson algebra} $D(2)$ and 
identifies with the polynomial subring $\bbF_2[u,v]\subset H^*(BA_4)$; 
see \cite[\S~III.2]{adem2013cohomology}.
 
From now on let $A$ be a graded, Noetherian commutative algebra over a field $k$ such that $A_0=k$.
 
\begin{definition}\label{def:ParameterIdeal}
 A \emph{parameter ideal} in $A$ is an ideal generated by a system of homogeneous parameters, i.e., by homogeneous elements 
$a_1,\ldots ,a_d$ of degree $>0$ such that $d$ is the Krull dimension of $A$ and $A/\langle a_1,\ldots, a_d\rangle$ 
is finite over $k$. 
\end{definition}

If a finite group $G$ acts on $A$, then the inclusion of invariants $A^G\subset A$ is a finite extension; see \cite[Theorem~2.3.1]{smith1995}. Thus $A^G$ and $A$ have the same Krull dimension. Since $H^*(B(\bbZ/2)^2)\cong \bbF_2[a,b]$ is of Krull dimension $2$, so are the invariant rings $H^*(B\SO(3))$ and  $H^*(BA_4)$. 

The algebra $H^*(BA_4)$ is an invariant ring 
for the group $C_3$ whose order is coprime to the characteristic of the field $\bbF_2$.
If the characteristic of $k$ does not divide the order of $G$, then an application of the Reynolds operator 
$\Rey\colon A\to A^G$ defined by
\[\Rey(a) = \frac{1}{|G|} \sum_{g\in G} ga\]
provides the following result.

\begin{lemma}[{see \cite[Lemma~2.6.10]{derksenkemper2015}}]\label{lem:contractionofideal}
Suppose that $\chara k$ does not divide $|G|$. For any ideal $I\subset A^G$, we have
$AI\cap A^G = I$. In particular $I\subset A^G$ is a parameter ideal if and only if $AI\subset A$ is a parameter ideal.
\end{lemma}

For the $\GL_2(2)$-action on $\bbF_2[a,b]$, the characteristic of $\bbF_2$ divides the group order. Nevertheless, we have a similar result.

\begin{lemma}\label{lem:contractionofidealinHSO3}
For any ideal $I\subset H^*(B\SO(3))$, we have $$H^*(BA_4)I\cap H^*(B\SO(3))=I.$$
\end{lemma}
\begin{proof}
The ideal $I$ is contained in its extension $H^*(BA_4)I$ and in $H^*(B\SO(3))$. We show the converse, i.e., that $H^*(BA_4)I\cap H^*(B\SO(3))\subset I$. An element of the extension $H^*(BA_4)I$ is a finite sum $X=\sum_i p_i f_i$ with $p_i\in H^*(BA_4)$ and $f_i\in I$. As a graded module over $H^*(B\SO(3))\cong \bbF_2[u,v]$, we have a direct sum
\[H^*(BA_4)\cong \bbF_2[u,v] \oplus \bbF_2[u,v]w.
\]
Thus each $p_i$ can be written uniquely as $p_i= q_i + q'_iw$ with $q_i,q'_i\in \bbF_2[u,v]$ and $X$ decomposes as
\[X= (\sum_i q_if_i) + (\sum_i q'_if_i)w
.\]
Suppose that $X\in H^*(B\SO(3))$. Then $\sum_i q'_i f_i=0$ and hence $X=\sum_i q_if_i$ lies in $I$.
\end{proof}

Recall that $A$ is Cohen-Macaulay if it has a parameter ideal generated by a regular sequence. We have the following:

\begin{lemma} The rings $H^* (BA_4)$ and $H^*(B\SO(3))$ are Cohen-Macaulay rings. 
\end{lemma}

\begin{proof} The polynomial algebra $H^*(B(\bbZ/2)^2)=\bbF_2[a,b]$ has a system of parameters given by the regular sequence $a$, $b$ and thus is Cohen-Macaulay. In a Cohen-Macaulay ring a sequence 
of $d=\dim A$ elements $a_1,\ldots,a_d$ is a system of homogeneous parameters if and only if 
it is a regular sequence; see e.g. \cite[Theorem~A.3.5]{neuselsmith2002}.  By \cite[Proposition~5.1.1]{neuselsmith2002}, two homogeneous 
elements $a_1,a_2$ in an invariant ring of $\bbF_2[a,b]$ form a regular sequence 
in $\bbF_2[a,b]$ if and only if they form a regular sequence in the invariant ring. 
Since $\{u, v\}$ is a system of parameters in $\bbF_2[a,b]$, it is a regular sequence 
in $\bbF_2[a,b]$, hence it is a regular sequence in $H^* (BA_4)$ and $H^* (B\SO(3))$.
It follows that these rings are Cohen-Macaulay.
\end{proof}

All three algebras $H^*(B(\bbZ/2)^2)$, $H^*(BA_4)$, and $H^*(B\SO(3))$ are unique factorization domains; see \cite[Theorem~2.11]{nakajima1982} for $H^*(BA_4)$. This gives the following:

\begin{lemma}\label{lem:Coprime} Let $I$ be an ideal in $H^* (BA_4)$, $H^*(B\SO(3))$, or $H^*(B(\bbZ/2)^2)$ generated by homogeneous elements $X$ and $Y$ of positive degrees. Then the following are equivalent:
\begin{enumerate} 
\item $I$ is a parameter ideal.
\item $X, Y$ form a regular sequence.
\item $X$ and $Y$ are coprime.
\end{enumerate}
\end{lemma}

\begin{proof}
Let $A$ denote any of the three cohomology rings. Since $A$ is Cohen-Macaulay and of Krull dimension $2$, the first two statements are equivalent; see e.g. \cite[Theorem~A.3.5]{neuselsmith2002}. By definition, elements $X, Y$ of positive degree form a regular sequence if and only if multiplication by $X$ is injective on $A$ and multiplication by $Y$ is injective on $A/\langle X\rangle$. Since $A$ does not have zero-divisors, the first condition holds. Since $A$ is a unique factorization domain, the kernel of multiplication by $Y$ on $A/\langle X\rangle$ is generated by $X$ divided by the greatest common divisor of $X$ and $Y$. Hence the second condition is equivalent to $X$ and $Y$ being coprime.
\end{proof}

Our goal is to classify all parameter ideals in $H^*(B\SO(3))$ and in $H^*(BA_4)$ that are closed under Steenrod operations.  
\begin{definition}
An ideal $I$ in the mod-$2$ cohomology of a space is \emph{Steenrod closed} if $\Sq(I)\subset I$, where $\Sq$ denotes the total Steenrod square.
\end{definition}

The total Steenrod square for $H^*(B(\bbZ/2)^2)=\bbF_2[a,b]$ is the ring homomorphism determined by
\[\Sq(a)=a^2 + a, \quad \Sq(b)=b^2 +b.
\]
Hence the total Steenrod squares of the generators $u,v, w$ for $H^*(BA_4)$ can be computed as follows:
\begin{align*}
\Sq(u) &=u+v+u^2,\\
\Sq(v) &=v+uv+v^2,\\
\Sq(w)&=w+u^2+u(v+w)+w^2.
\end{align*}

If $f\colon R \to S$ is a ring homomorphism and $I \subset S$ is an ideal in $S$, then the ideal $f^{-1} (I)$ is called contraction of $I$. For an ideal $J \subset R$, the ideal generated by the image $f(J)$ is called an extension of $J$.

\begin{lemma}\label{lem:steenrodclosedinextension}
An ideal $I\subset H^*(BA_4)$ is Steenrod closed if and only if its extension in $\bbF_2[a,b]$ is Steenrod closed. An ideal $I\subset H^*(B\SO(3))$ is Steenrod closed if and only if its extension in $H^*(BA_4)$ is Steenrod closed.
\end{lemma}

\begin{proof}
Consider the ring homomorphisms $$f \colon H^* (B\SO (3)) \to H^* (BA_4) \quad \text{and} \quad g \colon H^* (BA_4) \to \bbF_2 [a,b]$$ defined by inclusions. Since both $f$ and $g$ preserve Steenrod operations, Steenrod closed ideals in these rings are closed under extension and contraction; see \cite[Lemma~9.2.2]{neuselsmith2002}. By
\cref{lem:contractionofideal}, the contraction of the extension 
$\bbF_2 [a, b] I$, with respect to $g$, of an ideal $I\subset H^*(BA_4)$ is $I$ itself, 
and the analogous statement holds for contractions of extensions of ideals $I\subset H^*(B\SO(3))$ with respect to $f$
by \cref{lem:contractionofidealinHSO3}.  
\end{proof}

We will use repeatedly and often implicitly the following basic properties:

\begin{proposition}[{see \cite[Chapter~1]{walkerwood2018vol1}}] Let $R=\bbF_2[a,b]$. The Steenrod operations satisfy the following properties:
\begin{enumerate}
    \item $\Sq\colon R\to R$ commutes with the $\GL_2(2)$-action.
    \item  For $f\in R$ and $s,i\geq 0$,
    \[\Sq^i(x^{2^s}) = \begin{cases} (\Sq^j(x))^{2^s} \quad &\text{if } i =2^s j, \\
    0,\quad & \text{else.}\end{cases}\]
    \item The Cartan formula 
    \[\Sq^n(xy)=\sum_{i+j=n}\Sq^i(x)\Sq^j(y)\]
    holds. In particular, $\Sq^1$ is a derivation and $\Sq^1(x^2)=0$ for any $x$.
    \item  The sequence
    \[\langle a,b\rangle \xrightarrow{\Sq^1} \langle a,b\rangle \xrightarrow{\Sq^1}  \langle a,b\rangle\] 
    of $\bbF_2$-vector spaces is exact.
\end{enumerate}
\end{proposition}
\begin{remark}\label{rem:Sq1exact}
The induced sequence 
    \[\langle a,b\rangle\cap R^G \xrightarrow{\Sq ^1} \langle a,b\rangle \cap R^G 
    \xrightarrow{\Sq^1} \langle a,b\rangle\cap R^G\]
    is exact for $G=C_3$ since the fixed point functor is exact if the characteristic of the field is coprime to the order of the group. It is not exact for $G=\GL_2(2)$, since there is no element $x$ with $\Sq^1(x)=u^2$.
\end{remark}
We will use the following observations to detect invariant elements.
\begin{lemma}\label{lem:divisionlandsininvariantring}
Suppose that a finite group acts on an integral domain $A$ via ring homomorphisms.
\begin{enumerate}
    \item For any equality $x=\lambda y$ with $x,y\in A^G$, $y\neq 0$ and $\lambda\in A$, it follows that $\lambda\in A^G$.
    \item If the characteristic of $A$ is $2$, then square roots are unique. In particular, $x^2\in A^G$ for $x\in A$ implies $x\in A^G$.
\end{enumerate}
\end{lemma}
\begin{proof}\begin{enumerate}
\item For any $g\in G$, we have
\[\lambda y =x=gx = (g\lambda)(gy)=(g\lambda)y.\]
Since $A$ has no zero divisors, it follows that $\lambda=g\lambda$.
\item Let $x,y\in A$. If $0=x^2-y^2=(x-y)^2$, then $x=y$. If $x^2\in A^G$, then $x^2=g(x^2)=(gx)^2$ for all $g\in G$. Since square roots in $A$ are unique, it follows that $x=gx$ for all $g\in G$, i.e., $x\in A^G$.\qedhere
\end{enumerate}
\end{proof}

\section{Classification of twisted ideals}\label{sect:Twisted}
In this section, we classify the Steenrod closed parameter ideals in $H^*(BA_4)$ that have a system of parameters of the form $\{X,\Sq^1(X)\}$. We start with a lemma on parameter ideals.

\begin{lemma}\label{lem:lowergeneratorunique}
Let $I\subset H^*(BA_4)$ be a Steenrod closed parameter ideal. Choose a homogeneous system $\{X, Y\}$ of parameters for $I$, then $|X|\neq |Y|$ and the parameter of lower degree as well as the degrees of the parameters are independent of the choice.
\end{lemma}
\begin{proof}
Any Steenrod closed ideal generated by homogeneous elements of the same degree is of the form $\langle v^i\rangle$ by \cref{thm:oliver}
and thus cannot be a parameter ideal. Hence $X$ and $Y$ must have different degrees. If $|X|<|Y|$, then any homogeneous element of $I$ has degree at least $|X|$ and the homogeneous elements of degree $|X|$ form a one dimensional vector space over $\bbF_2$. Hence $X$ is independent of the choice of system of parameters. If $\{X,Y'\}$ is another system of parameters, then $Y'= \lambda Y+ \mu X$ for homogeneous coefficients $\mu, \lambda\in H^*(BA_4)$. The coefficient $\lambda$ is nonzero since $I$ is a parameter ideal. Hence $|Y|\leq |Y'|$ and analogously $|Y'|\leq |Y|$. Thus the degree of $Y$ is independent of the choice as well.
\end{proof}

\begin{lemma} \label{lem:Sq1Xneq0means} If $\langle X,Y\rangle$ is a Steenrod closed ideal in $H^*(BA_4)$ 
where $X$ and $Y$ are homogeneous elements with $|X|<|Y|$, then either $\Sq^1(X)=0$ or $Y=\Sq^1(X)$.
\end{lemma}
\begin{proof}
We can write $\Sq^1(X) = \lambda X+\nu Y$ for some elements $\lambda$ and $\nu$ in $H^*(BA_4)$. Since $H^1(BA_4)=0$, we have $\lambda=0$. The element $\nu$ is of degree $0$, so it is either 0 or 1.
\end{proof}

\begin{definition}\label{def:twistedideal}We call a Steenrod closed parameter ideal $I\subset H^*(BA_4)$ \emph{twisted} if it has a system of parameters of the form $\{X,\Sq^1(X)\}$. Otherwise it is called \emph{nontwisted}.  
\end{definition}

\begin{remark}\label{rem:twistedUniqueParameters}
If $I\subset H^*(BA_4)$ is twisted, then $I$ has a unique homogeneous system of parameters. Indeed, the parameter of lowest degree $X$ is unique by \cref{lem:lowergeneratorunique} and the only possibility for the second parameter is $\Sq^1(X)$ by \cref{lem:Sq1Xneq0means}.
\end{remark}
We will use the following observation for subsequent calculations.
\begin{lemma}\label{lem:kameko_decomposition}
    The $\bbF_2$-linear map $\bbF_2[a,b]^4 \to \bbF_2[a,b]$ that sends 
    $(x,y,z,t)$ to $x^2+ay^2+bz^2+abt^2$ is an isomorphism.
    
    Elements concentrated in even degrees correspond to tuples of the form $(x,0,0,t)$ and elements concentrated in odd degrees correspond to tuples of the form $(0,y,z,0)$.
\end{lemma}
\begin{proof}
Decompose the $\bbF_2$-vector space $\bbF_2[a,b]$ as
\[\bbF_2[a,b] = \bbF_2[a,b]^{\text{even},\text{even}} \oplus \bbF_2[a,b]^{\text{odd},\text{even}}\oplus \bbF_2[a,b]^{\text{even},\text{odd}}\oplus \bbF_2[a,b]^{\text{odd},\text{odd}},
\]
where $\bbF_2[a,b]^{\text{even},\text{even}}$ is spanned by the monomials such that the exponents of both $a$ and $b$ are even, $\bbF_2[a,b]^{\text{even},\text{odd}}$ is spanned by the monomials such that the exponent of $a$ is even and the exponent of $b$ is odd, etc.
Then each of $x\mapsto x^2$, $y\mapsto ay^2$, $z\mapsto bz^2$, and $t\mapsto abt^2$ is an $\bbF_2$-linear isomorphism from $\bbF_2[a,b]$ to the corresponding direct summand.

The elements concentrated in even degrees are the elements in 
$\bbF_2[a,b]^{\text{even},\text{even}}\oplus \bbF_2[a,b]^{\text{odd},\text{odd}}$.
The elements concentrated in odd degrees are the elements in 
$\bbF_2[a,b]^{\text{even},\text{odd}}\oplus \bbF_2[a,b]^{\text{odd},\text{even}}$.
\end{proof}

For cohomology classes in $H^*(B(\bbZ/2)^2)$ or $H^* (BA_4)$, we have the following observations (see \cite[Proof of Lemma~1]{Oliver}).

\begin{lemma}\label{lem:General} Suppose that $X$ is a homogeneous element of degree $n$ in $\bbF_2[a,b]$ which satisfies $\Sq^1(X)=0$. 
\begin{enumerate}
    \item If $n$ is even, then $X=x^2$ for some homogeneous $x\in \bbF_2[a,b]$.
    \item If $n$ is odd, then $X=vx^2$ for some homogeneous $x\in \bbF_2[a,b]$.
\end{enumerate}
If $X\in H^*(BA_4)$, then $x$ is an element of $H^*(BA_4)$ in either case.
\end{lemma}
\begin{proof} We use the observation from \cref{lem:kameko_decomposition}.
If $n$ is even, then we can write $X=x^2 +abt^2$ for some $x, t\in \bbF_2 [a,b]$. This gives
\[0=\Sq^1(x^2+abt^2) = (a^2b+ab^2)t^2,\]
which implies that $t=0$. So $X$ is the square of $x$.
If $X\in H^*(BA_4)=\bbF_2[a,b]^{C_3}$, then it follows that $x\in H^*(BA_4)$ by \cref{lem:divisionlandsininvariantring}.

If $n$ is odd, we can write $X=ay^2 +bz^2$. Computing $\Sq^1(X)$, we get \[0=\Sq^1(X)= a^2y^2+b^2 z^2=(ay+bz)^2,\]
which gives $ay=bz$. Hence $y=bx$ and $z=ax$ for some $x\in \bbF_2 [a, b]$. It follows that $X=ab^2 x^2 + a^2bx^2=vx^2$. By \cref{lem:divisionlandsininvariantring}, if $X\in H^*(BA_4)=\bbF_2[a,b]^{C_3}$, then $x\in H^*(BA_4)$.
\end{proof}

We introduce the following notation for the inverse of the map from \cref{lem:kameko_decomposition}.

\begin{definition}\label{def:kameko}
We write 
\[\kappa=(\kappa_1,\kappa_a,\kappa_b,\kappa_{ab})\colon \bbF_2 [a,b]\to \bbF_2 [a,b]^4\] 
for the $\bbF_2$-linear map that sends $f=x^2+ay^2+bz^2+abt^2$ to $(x,y,z,t)$.
\end{definition}

\begin{remark}\label{rem:kappaandsquares}
For each $*\in \{1,a,b,ab\}$, the map $\kappa_*$ satisfies $\kappa_*(xy^2)=\kappa_*(x)y$.
From the definition of $\kappa$ we immediately get
\[\kappa_{1}(x)=\kappa_{ab}(abx),\ 
\kappa_{a}(x)=\kappa_{ab}(bx),\text{ and }\
\kappa_{b}(x)=\kappa_{ab}(ax).
\]
\end{remark}
The restriction of $\kappa_{ab}$ to homogeneous elements of even degrees is the down Kameko map from \cite[Definition~1.6.2]{walkerwood2018vol1}.

We will now classify all twisted ideals as defined in \cref{def:twistedideal}. The classification result states that all parameters of twisted ideals are obtained from  $\{1, 0\}$ by a recursion formula. 

\begin{definition}\label{def:TwistedPairsSequence}
For $n\geq 1$, let $(x_n, y_n)$ be the sequence recursively defined by 
$(x_1, y_1)=(u,v)$ and 
$$x_n =ux_{n-1}^2 +y_{n-1}^2 \ \text{  and   }\ y_n=vx_{n-1}^2$$ for $n\geq 2$.
\end{definition}

Note that since $\Sq(u)=u+v+u^2$ and $\Sq(v)=v+uv+v^2$, the ideal $\langle u,v\rangle$ is a Steenrod closed parameter ideal and since $\Sq^1(u)=v$, it is a twisted ideal.
For $n=2$, we have $(x_2, y_2)=(u^3+v^2, vu^2)$, and 
by direct calculation one can show that $\langle x_2, y_2\rangle$ is a twisted ideal.
Our first observation is that all pairs $(x_n, y_n)$ obtained this way are parameters of a twisted ideal for all $n\geq 1$.
 
\begin{lemma}\label{lem:SeqAreGood}
Let $(\mu_n)$ be the sequence recursively defined as $\mu_0=1$ and \[\mu_n = (1+u+v)\mu_{n-1} ^2 +x_n^2\] for $n\geq 1$.
For the sequence of pairs $(x_n,y_n)$ defined in \cref{def:TwistedPairsSequence}, we have
\begin{align*} \Sq(x_n) &= \mu_{n-1} (x_n+y_n)+x_n^2,\\
    \Sq(y_n) & = v\mu_{n-1} x_n + ((u+1) \mu_{n-1} +x_n) y_n +y_n^2.
\end{align*}
In particular for all $n\geq 1$, the pairs $(x_n, y_n)$ generate a Steenrod closed parameter ideal and $\Sq^1(x_n)=y_n$, thus $\langle x_n,y_n\rangle$ is twisted.
\end{lemma}
\begin{proof}
From the recursion formula in \cref{def:TwistedPairsSequence}, it is clear that $x_n$ and $y_n$ are coprime for all $n\geq 1$. Since their degrees are positive, they generate parameter ideals.  
We have that $\Sq^1(x_1)=\Sq^1 (u)=v=y_1$ and $\Sq^1(x_n) = \Sq^1(u)x_{n-1}^2 = vx_{n-1}^2 = y_n$ for $n\ge 2$.

We show that the formulas for $\Sq(x_n)$ and $\Sq(y_n)$ hold by induction on $n$. The base case $n=1$ follows immediately by inserting $x_1=u$, $y_1=v$, and $\mu_0=1$ to the equations. 
Now assume the above equations hold for some pair $(x_n, y_n)$ for $n\geq 1$. We will show that they also hold for the pair $(x_{n+1}, y_{n+1})$. To make the equations simpler, we write $(x,y)=(x_n, y_n)$, $(X,Y)=(x_{n+1}, y_{n+1})$, 
$\mu=\mu_{n-1}$, and $\eta=\mu_n$. So we have $X=ux^2+y^2$, $Y=vx^2$, and $\eta=(1+u+v) \mu ^2 + x^2$. This gives  
\begin{align*}
\Sq(X)=&\Sq(ux^2+y^2)= (u+v+u^2)\Sq(x)^2+\Sq(y)^2  \\
=& (u+v+u^2)(\mu (x+y)+x^2)^2+ (v\mu x)^2+ ((u+1) \mu +x)^2 y^2  +y^4 \\
=& (u+v)(\mu ^2 x^2 +\mu^2 y^2 +x^4)+ u^2 \mu^2 x^2 + u^2 \mu^2 y^2+ u^2 x^4+ v^2 \mu^2 x^2 \\
&+(u^2+1) \mu^2 y^2 +x^2 y^2 +y^4 \\ 
=& (\mu^2 + u\mu^2+v\mu^2+x^2)(ux^2+y^2+vx^2) +(ux^2+y^2)^2 \\
=&\eta (X+Y) +X^2.
\end{align*}
Similarly for $\Sq(Y)$ we have
\begin{align*}
\Sq(Y) =& \Sq( vx^2)= (v+uv+ v^2 ) (\mu (x+y) +x^2 )^2  \\
=&v\mu^2x^2+ v\mu^2y^2 + vx^4 + uv \mu^2 x^2+ uv\mu^2y^2 + uv x^4  \\
&+ v^2 \mu^2 x^2+ v^2 \mu^2 y^2+ v^2x^4 \\
=&v(\mu^2 +u \mu^2 +v \mu^2 +x^2) (ux^2 +y^2)+(1+u^2+uv+v)\mu^2vx^2  \\
&+(x^2+y^2) vx^2 + v^2x^4 \\
=&v\eta X+ ((u+1) \eta +X)Y  +Y^2. \qedhere
\end{align*}
\end{proof} 
We record the following calculation for the classification of Steenrod closed parameter ideals in \cref{sect:MixedCase}.
\begin{lemma}\label{lem:xymumodv}
In $\bbF_2 [u,v]/\langle v\rangle=\bbF_2[u]$, for every $n \geq 1$, we have $x_n\equiv u^{2^n -1}$, $y_n \equiv 0$, and 
  for $n \geq 0$, $$ \mu_n \equiv  \sum_{i=0}^{2^{n+1}-2} u^i = (1+u^{2^ {n+1}-1})/(1+u).$$
\end{lemma}
\begin{proof}
Since $y_1=v$ and $y_n=vx_{n-1}^2$, it follows that $y_n\equiv 0$ modulo $\langle v\rangle$. The other formulas hold by a straightforward induction.
\end{proof}

In addition to the inductive definition of $x_n$, $y_n$, we can express them explicitly as elements in $\bbF_2[a,b]$.
\begin{remark} 
For $n\geq 1$, let $m=2^{n+1}-2.$ Then for all $n \geq 1$, 
$$x_n= \sum _{i=0} ^m a^i b^{m-i} =(a^{m+1} +b^{m+1})/(a+b), $$
$$ y_n =ab \sum _{i=0} ^{m-1} a^i b^{m-i}=ab(a^m+b^m)/(a+b).$$
Note that for $n=1$, we have
$x_1=u=a^2 +ab+b^2=(a^3+b^3)/(a+b)$ and $y_1=v=ab(a+b)=ab(a^2+b^2)/(a+b)$. 
The general case can be proved easily by induction. \end{remark}

Now we prove that every twisted ideal $\langle X,\Sq^1(X)\rangle$ satisfies $X=x_n$ for some $n \geq 1$.
We do this by considering two separate cases.

 \begin{lemma}\label{lem:xsq1xeven} Let $X\in H^*(BA_4)$ be a  homogeneous element of even degree such that $\langle X,\Sq^1(X)\rangle$ is Steenrod closed. Then $\Sq^1(X)=0$ or $X=uy^2+\Sq^1(y)^2$ for some $y\in H^*(BA_4)$ and $\langle y,\Sq^1(y)\rangle$  is also Steenrod closed.
\end{lemma}
 
\begin{proof} Recall that $H^*(BA_4)\cong \bbF_2[a,b]^{C_3}$, where a generator of $C_3$ acts on $\bbF_2[a,b]$ via the automorphism $\varphi$ given by $\varphi(a)=b$, $\varphi(b)=a+b$. Write $X$ in the form $X=x^2+aby^2$. Then
\[x^2+aby^2=X=\varphi(X)= \varphi(x)^2+b(a+b)\varphi(y)^2  = \varphi(x)^2+ b^2\varphi(y)^2+ab \varphi(y)^2.\]
Since the way to express $X$ in the form above is unique, we get $\varphi(y)=y$ and hence $y\in H^*(BA_4)$.
Using the Cartan formula, we get
\begin{align*}
\Sq^2(X)&=\Sq^1(x)^2+a^2b^2y^2+ab\Sq^1(y)^2.
\end{align*}
Since $\Sq^2(X)\in \langle X,\Sq^1(X)\rangle$ by assumption, we can write $\Sq^2(X)$ as a linear combination of $X$ and $\Sq^1(X)$ with coefficients in $H^*(BA_4)$. For degree reasons, the coefficient of $\Sq^1(X)$ has to be zero. We thus get
\[\Sq^2(X) \in \{0, uX\}.\]

If $\Sq^2(X)=0$, then
\begin{align*}
0&=\kappa_{1}(\Sq^2(X))=\Sq^1(x)+aby,\\
0&=\kappa_{ab}(\Sq^2(X))=\Sq^1(y).
\end{align*}
Applying $\Sq^1$ on both sides of the first equation yields 
\[0=\Sq^1(\Sq^1(x)+aby) = \Sq^1(\Sq^1(x))+\Sq^1(ab)y+ab\Sq^1(y)=0+vy+0=vy.\]
As $\Sq^1(X)=vy^2$, it follows that $\Sq^1(X)=0$. This completes the first case.

Now assume that $\Sq^2(X)=uX$. Then we have
\begin{align*}
    \Sq^1(x)^2+a^2b^2y^2+ab\Sq^1(y)^2&=(a^2+ab+b^2)(x^2+aby^2)\\
    &=(a^2+b^2)x^2+a^2b^2y^2+ab(x^2+a^2y^2+b^2y^2).    
\end{align*}
Applying $\kappa_1$ and $\kappa_{ab}$ yields
\begin{align*}
    \Sq^1(x)+aby &= (a+b)x+aby,\\
    \Sq^1(y)&=x+ay+by.
\end{align*}
Solving the last equation for $x$ and inserting it into the definition of $X$ gives
\[X=x^2+aby^2=(a^2+b^2)y^2+\Sq^1(y)^2+aby^2 = uy^2+\Sq^1(y)^2.\]
So $X$ has indeed the required form. 

It remains to show that the ideal $\langle y,\Sq^1(y)\rangle$ in $ H^*(BA_4)$ is Steenrod closed. 
By \cref{lem:steenrodclosedinextension}, it suffices to show that the extension $\langle y,\Sq^1(y)\rangle\subset \bbF_2[a,b]$ is Steenrod closed. Since the ideal $\langle X, \Sq^1 (X)\rangle$ is Steenrod closed, we can find coefficients $\lambda,\mu$ such that
$\Sq (X)=\lambda X +\mu \Sq^1 (X)$. 
This gives 
\begin{equation}\label{SqEqn}
\Sq (u) \Sq (y)^2 + \Sq (\Sq ^1 (y) )^2 =\lambda uy^2 +\lambda \Sq^1(y)^2+\mu v y^2.
\end{equation}
Let $*\in \{1,a,b,ab\}$. Using \cref{rem:kappaandsquares}, the map $\kappa_*$ applied to the right-hand side of \eqref{SqEqn} 
yields
\[\kappa_*(\lambda u)y + \kappa_*(\lambda) \Sq^1(y) +\kappa _* (\mu v) y,\]
which is an element of $\langle y, \Sq^1(y)\rangle$.
Since $$\Sq(u)= u+v+u^2= a^2 +ab +b^2+a^2 b +ab^2+ a^4 +(ab)^2+b^4,$$ $\kappa_{ab}$ applied to the left-hand side of \eqref{SqEqn} yields $\Sq(y)$. Hence $\Sq(y)$ is in the ideal $\langle y,\Sq^1(y)\rangle$. Applying $\kappa_1$ to the left-hand side yields
\[ (a+b +a^2+ab+b^2 ) \Sq(y) +\Sq (\Sq^1 (y) ). \]
It follows that $\Sq(\Sq^1(y))$ also lies in the ideal $\langle y,\Sq^1(y)\rangle$.
\end{proof}

If $\Sq^1(X)=0$, then $I=\langle X,\Sq^1(X)\rangle=\langle X\rangle$  is not a parameter ideal. 

\begin{lemma}\label{lem:xsq1xodd}Let $X\in H^*(BA_4)$ be a homogeneous element of odd degree such that $\langle X, \Sq^1(X)\rangle$ is Steenrod closed. Then both $X$ and $\Sq^1(X)$ are divisible by $v$, hence in this case $\langle X, \Sq^1(X) \rangle $ is not a parameter ideal.
\end{lemma}
\begin{proof}
If an element $X$ is divisible by $v$, e.g. $X=vz$, we get that 
\[\Sq^1(X)=\Sq^1(vz)=\Sq^1(v)z+v\Sq^1(z)=0+v\Sq^1(z)\]
is also divisible by $v$. So it is enough to show that $X$ is divisible by $v$. The last part will follow from
\cref{lem:Coprime}. Since $0$ is divisible by $v$, we may assume $X\neq 0$.

Write $X=ax^2+by^2$ and thus $\Sq^1(X)=a^2x^2+b^2y^2$. Now use again that $\Sq^2(X) \in \{0,uX\}$ for degree reasons. The Cartan-formula shows
\[\Sq^2(X)= a\Sq^1(x)^2+b\Sq^1(y)^2.\]
Let us first look at the easier case of $\Sq^2(X)=uX$. In this case we have
\[a\Sq^1(x)^2+b\Sq^1(y)^2=(a^2+ab+b^2)(ax^2+by^2).\]
Applying $\kappa_a$ and $\kappa_b$ yields 
\begin{align*}
    \Sq^1(x)=&ax+bx+by,\\
    \Sq^1(y)=&ax+ay+by.
\end{align*}
Now we can use that $\Sq^1\circ \Sq^1=0$:
\begin{align*}
0&=\Sq^1(ax+bx+by) \\
&=a^2x+a\Sq^1(x)+b^2x+b\Sq^1(x)+b^2y+b\Sq^1(y)\\
&=a^2x+(a+b)(ax+bx+by)+b^2x+b^2y+b(ax+ay+by)\\
&=abx+b^2y.
\end{align*}
Thus we get that $ax=by$. Hence $x$ is divisible by $b$ and we can substitute $x=bz$ to obtain $y=az$. Inserting this in the definition of $X$ yields
\[X=(ab^2+a^2b)z^2 = vz^2.\]
Hence $X$ is divisible by $v$. In fact in this case we have $\Sq^1 (X)=0$.

Let us now look at the case of $\Sq^2(X)=0$. From the equation
\[\Sq^2(X)= a\Sq^1(x)^2+b\Sq^1(y)^2\]
we get $\Sq^1(x)=0$ and $\Sq^1(y)=0$.

If $|x|$ is odd, this means by \cref{lem:General} that  $x=vz^2$ and $y=vt^2$
for some $z, t \in \bbF_2 [a,b]$. Thus both $X$ and $\Sq^1(X)$ would be divisible by $v$ in $\bbF_2[a,b]$ and hence also in $H^*(BA_4)$ by \cref{lem:divisionlandsininvariantring}.

If $|x|$ is even, then consider the ideal $J$ generated by $ax+by$ and $\Sq^1(ax+by)=a^2x+b^2y$. We want to 
show that $J$ is again Steenrod closed. Since 
\[
a^{n+1}x+b^{n+1}y = (a+b)(a^nx+b^ny)+ab(a^{n-1}x+b^{n-1}y)\]
it follows inductively that $J$ contains all elements of the form $a^nx+b^ny$ for $n\ge 1$ and thus also all elements of the form 
\[a^nby+ab^nx = a^{n+1}x+b^{n+1}y+(a^n+b^n)(ax+by).\]

Since $\Sq(X)=\Sq(ax^2+by^2)$ is in $\langle X, \Sq^1 (X)\rangle=\langle ax^2+by^2,a^2x^2+b^2y^2\rangle$, there exist homogeneous elements $\lambda,\mu$ such that 
\[(a+a^2)\Sq(x)^2+(b+b^2)\Sq(y)^2=\Sq(ax^2+by^2)=\lambda (ax^2+by^2)+\mu (a^2x^2+b^2y^2).\]
Applying $\kappa_a$ and $\kappa_b$ to both sides of the equation yields
\begin{align*}
\Sq(x)&=\kappa_1(\lambda)x+\kappa_{a} (\lambda b) y+\kappa_a (\mu)(ax+by),\\
\Sq(y)&=\kappa_{b} (\lambda a) x+\kappa_1 (\lambda)y+\kappa_b (\mu)(ax+by).
\end{align*} 
Since $\kappa _a (\lambda b)=\kappa _{ab} (\lambda) b$ and $\kappa _b (\lambda a)=\kappa _{ab} (\lambda) a$, 
we get that for any $n\ge 1$: 
\begin{align*}
    a^n\Sq(x)+b^n\Sq(y)=&\kappa_1 (\lambda)(a^nx+b^ny)+\kappa_{ab}(\lambda)(a^nby+ab^nx)\\
    &+(a^n\kappa_{a}(\mu)+b^n\kappa_{b}(\mu))(ax+by)
\end{align*}
is in $J$.
Since $\Sq(ax+by)=(a+a^2)\Sq(x)+(b+b^2)\Sq(y)$ and $\Sq(a^2x+b^2y)=(a^2+a^4)\Sq(x)+(b^2+b^4)\Sq(y)$, we see that 
$J$ is Steenrod closed.

Note that $|X|>1$ since $H^1(BA_4)=0$. Thus the degree of $ax+by$ is smaller than the degree of $X$. Since $ax+by$ is the square root of $\Sq^1(X)$ it is also in $H^*(BA_4)$ by \cref{lem:divisionlandsininvariantring}. It follows by induction that $v=ab(a+b)$ divides both $ax+by$ and $a^2x+b^2y$.

We show that $v$ divides $X=ax^2+by^2$. Since $ax^2+by^2= (x+y)(ax+by) + xy(a+b)$ and $v|(ax+by)$, it suffices to show that $v|xy(a+b)$. Since 
\[v | (a^2x+b^2y)(x+y)-(ax+by)^2 = xy(a^2+b^2)\]
and $v=ab(a+b)$ it follows that $ab|xy$ and hence $v| xy(a+b)$.
\end{proof}

We summarize the results of this section in the following theorem.

\begin{theorem}\label{thm:classification_twisted_pairs}
Let $X$ be a homogeneous element in the cohomology ring $H^*(BA_4)$. The ideal $\langle X, \Sq^1(X)\rangle$ is a Steenrod closed parameter ideal if and only if there exists an $n\geq 1$ such that $X=x_n$, where $x_n$ is defined recursively via $x_1=u$ and $x_{n+1} = ux_n^2+\Sq^1(x_n)^2$ for $n\geq 1$. 
\end{theorem}
\begin{proof} By \cref{lem:SeqAreGood}, we know that
the ideal $\langle x_n,\Sq^1(x_n)\rangle$ is a Steenrod closed parameter ideal for any $n\ge 1$. Conversely, 
let $X \in H^* (BA_4)$ be a homogeneous element such that $\langle X,\Sq^1(X)\rangle$ is a Steenrod closed parameter ideal. In particular, we have $\Sq^1(X)\neq 0$. We show by induction on the degree of $X$ that it is of the given form. If $|X|$ is odd, then \cref{lem:xsq1xodd} tells us that $v$ divides $X$, and hence $X$ and $\Sq^1(X)$ cannot be coprime. If $|X|$ is even, then it follows from \cref{lem:xsq1xeven} that
$X=uY^2+\Sq^1(Y)^2$ for some $Y$ such that $\langle Y,\Sq^1(Y)\rangle$ is also Steenrod closed. If $Y$ and $\Sq^1(Y)$ had a common divisor, the same divisor would also divide $X$ and $\Sq^1(X)$. So
$Y$ and $\Sq^1(Y)$ are coprime. If $|Y|=0$, then $X=u=x_1$. If $|Y|>0$, then $\langle Y,\Sq^1(Y)\rangle$ is a parameter ideal and hence by induction assumption we have $Y=x_n$ for some $n\geq 1$ and thus $X=x_{n+1}$. 
\end{proof}

\section{Observations on nontwisted ideals}\label{sec:obsNonTwisted}
In this section we are interested in nontwisted Steenrod closed parameter ideals of $H^*(BA_4)$. By \cref{lem:Sq1Xneq0means}, these are the Steenrod closed parameter ideals such that $\Sq^1$ of the parameter of lower degree is zero. The results will be used in the following two sections to classify all Steenrod closed parameter ideals.

\begin{lemma}\label{lem:sq1xeq0sq1yeq0} Assume that $\langle X,Y\rangle$ is a Steenrod closed parameter ideal in $H^*(BA_4)$ where $X$ and $Y$ are homogeneous elements with $|X|<|Y|$ and 
that $\Sq^1(X)=0$. There exists a homogeneous element $Y'\in H^*(BA_4)$ with $\langle X,Y\rangle=\langle X,Y'\rangle$ and $\Sq^1(Y')=0$. 
\end{lemma}
\begin{proof}
Since $H^1(BA_4)=0$, we have $\Sq^1(Y)=\lambda X$ for some $\lambda\in H^{|Y|+1-|X|}(BA_4)$. Since $(\Sq^1)^2=0$ and $\Sq^1(X)=0$, we obtain
\[0=\Sq^1(\lambda X) = \Sq^1(\lambda)X+\lambda\Sq^1(X)=\Sq^1(\lambda)X.\]
Since $H^*(BA_4)$ does not have zero divisors and $X\neq 0$, we have $\Sq^1(\lambda)=0$ and by exactness of $\Sq^1$, there exists $\lambda'\in H^{|Y|-|X|}(BA_4)$ such that $\Sq^1(\lambda')=\lambda$. 

Setting $Y'=Y+\lambda'X$ provides a homogeneous element such that $\langle X,Y\rangle =\langle X, Y'\rangle$ and
$\Sq^1(Y+\lambda'X) = \lambda X+\lambda X=0$.
\end{proof}

We will be interested in Steenrod closed parameter ideals generated by squares. Therefore we introduce the following notation.

\begin{definition}[{\cite[\S2.6]{derksenkemper2015}}]\label{def:squaringgeneratorsofideal}
For an ideal $I$ of a commutative ring $R$ of characteristic $2$, we define $I^{[2]}$ to be the ideal generated by all $x^2$ with $x\in I$.
\end{definition}
If $I=\langle x,y\rangle$, then $I^{[2]}=\langle x^2, y^2\rangle$. Note that $I^{[2]}$ differs from $I^2=\langle x^2,xy,y^2\rangle$ in general.

\begin{lemma}\label{lem:ImapstoI[2]_injective}
The operation $I\mapsto I^{[2]}$ is injective on parameter ideals in $H^*(BA_4)$.
\end{lemma}
\begin{proof}
Let $I=\langle x,y\rangle$ and $J=\langle x',y'\rangle$ be parameter ideals in $H^*(BA_4)$ with $I^{[2]}=J^{[2]}$. By symmetry, it is enough to show that $J\subset I$. If $|x|=|y|$, then $|x'|=|y'|$. Hence $(x')^2$ and $(y')^2$ are $\bbF_2$-linear combinations of $x^2$ and $y^2$. In particular, the coefficients are squares themselves, say $(x')^2=\alpha^2 x^2 +\beta^2 y^2$. It follows that $x'=\alpha x + \beta y$. Similarly, $y'\in I$ and hence $J\subset I$.

If $|x|\neq |y|$, then after renaming the parameters, we may assume that $|x|<|y|$ and $|x'|<|y'|$. For degree reasons and since we work over $\bbF_2$, we have $(x')^2=x^2$ and $(y')^2= y^2+\lambda x^2$ for some homogeneous element $\lambda\in H^*(BA_4)$. By \cref{lem:divisionlandsininvariantring}, squares are unique in $H^*(BA_4)$, thus $x^2 =(x')^2$ implies $x'=x$.

The equation $(y')^2=y^2+\lambda x^2$ gives that $x^2$ divides $(y'+y)^2$.
Since $H^*(BA_4)$ is a unique factorization domain, we obtain that $x$ divides $y+y'$. Let $\mu \in H^* (BA_4)$ be such that
$y+y'=\mu x$. Then $y'=y+\mu x$ and $x'=x$, hence $J\subset I$.
\end{proof}

\begin{lemma}\label{lem:EvenEvenCase} Let $I$ be an ideal in $H^*(BA_4)$.  Then $I^{[2]}$ is Steenrod  closed if and only if $I$ is Steenrod closed. Moreover, $I^{[2]}$ is a parameter ideal if and only if $I$ is a parameter ideal.
\end{lemma}
\begin{proof}
If $I$ is Steenrod closed, then $I^{[2]}$ is Steenrod closed since $I^{[2]}$ is the extension of the ideal $I$ with respect to the Frobenius homomorphism on $H^*(BA_4)$ and the Frobenius homomorphism commutes with $\Sq$.

Conversely, suppose that $I^{[2]}$ is Steenrod closed. We show that $I$ is Steenrod closed. If $x\in I$ then $\Sq(x^2)\in I^{[2]}$ by assumption and we can find a linear combination \[\Sq(x)^2=\Sq(x^2)= \sum_i \lambda_i y_i^2\]
with $y_i\in I$. Applying $\kappa_1$ yields with the help of \cref{rem:kappaandsquares}
\[\Sq(x)=\sum_i \kappa_1(\lambda_i)y_i.\]
Since the coefficients $\kappa_1(\lambda_i)\in \bbF_2[a,b]$ need not be in $H^*(BA_4)=\bbF_2[a,b]^{C_3}$, this only shows that the extension of $I$ to $\bbF_2[a,b]$ is Steenrod closed. Nevertheless, it follows that $I$ is Steenrod closed by \cref{lem:steenrodclosedinextension}.

The second statement holds since two homogeneous elements $x,y\in H^*(BA_4)$ are coprime if and only if $x^2$ and $y^2$ are so.
\end{proof}

The following is immediate from the above results.
\begin{proposition}\label{pro:square_of_good_generators} If $I$ is a Steenrod closed parameter ideal in $H^*(BA_4)$ generated by parameters of even degrees, then $I=J^{[2]}$ for some Steenrod closed parameter ideal $J$.
\end{proposition}
\begin{proof}
Let $X$ and $Y$ be parameters of $I$ with $|X|<|Y|$. Since $X$ and $Y$ have even degrees, $I$ cannot be twisted. 
By \cref{lem:Sq1Xneq0means}, we have for a nontwisted $I=\langle X,Y\rangle$ that $\Sq^1(X)=0$ and 
by \cref{lem:sq1xeq0sq1yeq0}, we can pick the second parameter $Y'$ such that $\Sq^1(Y')=0$. 
By \cref{lem:General}, there exist $x,y$ such that $X=x^2$ and $Y'=y^2$. Hence if we take $J=\langle x,y\rangle$, 
then $I=J^{[2]}$. The lemma now follows from \cref{lem:EvenEvenCase}.
\end{proof}

Now we consider the case where one of the degrees $|X|$ and $|Y|$ is odd. For this case the following lemma is useful.

\begin{lemma}\label{lem:EvenOddCase}
Let $x,y\in H^*(BA_4)$ be homogeneous elements such that $|x|\neq 0$ and $\langle vx^2, y^2\rangle$ is a Steenrod closed parameter ideal. Then $\langle x,y\rangle$ is a Steenrod closed parameter ideal.
\end{lemma}
\begin{proof}
If $\langle vx^2,y^2\rangle$ is a parameter ideal, then so is $\langle x,y\rangle$ since the former is contained in the latter. It remains to check that $\langle x,y\rangle$ is closed under the Steenrod operations. By \cref{lem:steenrodclosedinextension}, it suffices to show that its extension $J$ in $\bbF_2[a,b]$ is Steenrod closed.
Let us first look at 
\[\Sq(y)^2 = \Sq(y^2)= \lambda vx^2+\mu y^2.\]
Applying $\kappa_1$ yields
\[\Sq(y) = \kappa_1(\lambda v)x+\kappa_1(\mu)y\]
and thus $\Sq(y)$ lies in $J$. 
The trickier part is to show that $\Sq(x)$ belongs to $J$. Since $\langle vx^2, y^2 \rangle$ 
is a Steenrod closed parameter ideal, we have
\[v(1+u+v)\Sq(x)^2 = \Sq(vx^2)=\alpha v x^2+\theta y^2\]
for some $\alpha$ and $\theta$.
Since $y$ is coprime to $vx^2$, $\theta$ has to be divisible by $v$. We thus get with $\theta=v\theta '$:
\begin{align}\label{eq:something}(1+u+v)\Sq(x)^2 = \alpha x^2+\theta ' y^2.\end{align}
Recall that 
\[1+u+v = 1+a^2+ab+b^2 + a^2b+ab^2=(1+a^2+b^2)+ab^2+a^2b+ab\]
and thus $\kappa_{ab}(1+u+v)=1$. Applying $\kappa_{ab}$ to \eqref{eq:something}
yields $\Sq(x)=\kappa_{ab}(\alpha )x+\kappa_{ab}(\theta' )y$ and thus $\Sq(x)\in J$.
\end{proof}
Recall from \cref{lem:lowergeneratorunique} that the degrees of parameters generating a Steenrod closed parameter ideal are independent of the system of parameters of the ideal.
\begin{proposition}\label{pro:generatorsoddeven}
For a nontwisted Steenrod closed parameter ideal $\langle X,Y\rangle$ in $H^*(BA_4)$ the following hold:
\begin{enumerate}
\item\label{item:notbothodd} Both $|X|$ and $|Y|$ cannot be odd.
\item\label{item:evenodd>3} If one of $|X|$, $|Y|$ is odd and $>3$, then there exists a Steenrod closed parameter ideal $\langle x, y\rangle$ such that $\langle X,Y\rangle =\langle vx^2,y^2\rangle$.
\item\label{item:evenodd3} If one of $|X|$, $|Y|$ is $3$, then $\langle X,Y\rangle =\langle v, y^2\rangle$ for some homogeneous element $y\in H^*(BA_4)$.
\end{enumerate}
\end{proposition}

\begin{proof} We may assume that $|X|<|Y|$. Since the ideal $\langle X, Y \rangle$ is nontwisted, by \cref{lem:Sq1Xneq0means} we have  $\Sq^1(X)=0$, and by \cref{lem:sq1xeq0sq1yeq0} we can pick the second generator $Y'$ such that $\Sq^1(Y')=0$. By \cref{lem:General},   
if both $|X|$ and $|Y|$ were odd, then $X$ and $Y'$ would be divisible by $v$ and thus they cannot form a system of parameters. 
If only one of the degrees of $X$ or $Y$ is odd, then $\langle X,Y\rangle =\langle vx^2,y^2\rangle$ by \cref{lem:General}. We conclude that \eqref{item:evenodd3} holds and \eqref{item:evenodd>3} follows from \cref{lem:EvenOddCase}.
\end{proof}

Usually it takes a lengthy computation to check if two elements generate a Steenrod closed parameter ideal. The next lemma provides a helpful criterion in a special case.

\begin{lemma}\label{lem:Condition} Suppose $\langle X,Y\rangle$ is a Steenrod closed ideal generated by homogeneous elements $X, Y$ in $H^*(BA_4)$ and let $k\geq 1$. Then $\langle v^k X,Y\rangle$ is Steenrod closed if and only if there exist coefficients $\lambda,\mu$ with 
\[\Sq(Y)=\lambda X+\mu Y\]
such that $v^k$ divides $\lambda$.
Furthermore if $X,Y$ are coprime, this is equivalent to asking that for any choice of coefficients $\lambda,\mu$ as above, we have $\lambda\in \langle v^k,Y\rangle$
\end{lemma}
\begin{proof}
Suppose that there are such coefficients and let $\lambda=v ^k \lambda'$. We then have
\begin{align*}
\Sq(Y)&=\lambda' v^k X+\mu Y \in \langle v^k X,Y\rangle,\\
\Sq(v^k X)&=v^k(1+u+v)^k \Sq(X)\in v^k \langle X,Y\rangle\subset \langle v ^k X,Y\rangle
\end{align*}
since $\Sq(X)\in \langle X,Y\rangle$. Conversely, assume that $\langle v^k X,Y\rangle$ is Steenrod closed. We then have
\[\Sq(Y)=\lambda'v^k X+\mu Y\]
for some $\lambda',\mu$. This gives exactly the coefficients as above.

For the second part note that if $X,Y$ are coprime, then any other choice of coefficients $\lambda',\mu'$ is of the form
\[\lambda' = \lambda +pY\mbox{ and }\mu' =\mu+pX.\]
This implies the second part of the statement.
\end{proof}

As a consequence we conclude the following.
\begin{corollary}\label{AllPowers}
Let  $\langle X,Y \rangle$ be Steenrod closed and $k\geq 1$. If 
$\langle v^k X,Y \rangle$ is Steenrod closed then $\langle v^i X,Y\rangle$ is Steenrod closed 
for every $i \leq k$
\end{corollary}

\section{Classification of fibered ideals}\label{sect:FiberedCase}

\begin{definition}\label{def:fibered}
A Steenrod closed parameter ideal $I\subset H^*(BA_4)$ is called \emph{fibered}, if it has a system of parameters $\{X,Y\}$ such that $\langle X\rangle$ is Steenrod closed.
\end{definition}
Note that we do not assume here that $X$ is the generator of smaller degree. By \cref{thm:oliver}, if $I=\langle X \rangle$ is a proper ideal in $H^* (BA_4)$ closed under Steenrod operations then there is a $k\geq 1$ such that $X= v^k$. This shows that fibered ideals are always of the form $\langle v^k,Y\rangle$. 
In this section we show that all fibered ideals are of the form $\langle v^k, u^l \rangle$ where $k$ and $l$ satisfy the following condition: if $l= 2^t c$ with $c$ odd, then $k\leq 2^t$. This condition comes from an observation due to Meyer and Smith.

\begin{theorem}[{\cite[Theorem~V.1.1]{meyersmith2003}}]\label{thm:MeyerSmith} The ideal $\langle v^k, u^l \rangle \subset H^*(BA_4)$ is Steenrod closed if and only if $l= 2^t c$ with $c$ odd and $k\leq 2^t$. 
\end{theorem}

In the proof of our results in this section we do not use this theorem. Moreover, this theorem
follows form the main theorem of this section (\cref{thm:classification_fibered}).

\begin{lemma}\label{lem:vX2Y2normalform}
Let $X,Y\in H^*(BA_4)$ be two coprime, homogeneous elements such that $|X|<|Y|$ and let $\lambda \in H^{|Y|-|X|}(BA_4)$ be any homogeneous element. Then 
\begin{align*}
\langle vX^2,(Y+\lambda X)^2\rangle &=\begin{cases}\langle vX^2,Y^2\rangle &\lambda\in \langle v\rangle \\
\langle vX^2,Y^2+u^{|Y|-|X|}X^2\rangle & \lambda \notin \langle v\rangle\end{cases},\\
\langle v(Y+\lambda X)^2,X^2\rangle&= \langle vY^2,X^2\rangle.
\end{align*}
\end{lemma}
\begin{proof}
First note that adding multiples of $v$ to $\lambda$ does not change the ideal:
\[\langle vX^2,(Y+(\lambda+\lambda'v)X)^2\rangle = \langle vX^2,Y^2+\lambda^2X^2+\lambda'^2v^2X^2\rangle=\langle vX^2,Y^2+\lambda^2X^2\rangle.\]
So the case $\lambda \in \langle v\rangle$ is obvious. 

Now assume that $\lambda \not \in \langle v \rangle$. Write $\lambda$ in the form $\lambda=p+wp'$ where $p,p'\in \bbF_2[u,v] \subset H^*(BA_4)$. We may leave out all monomials divisible by $v$ and thus assume that $p,p'\in \bbF_2[u]$. Hence $p$ and $p'$ have even degree and for degree reasons, only one of them can be nonzero. They cannot both be zero, since then $\lambda$ would be divisible by $v$. 
If $p\neq 0$, we have $p=u^{(|Y|-|X|)/2}$ since $\lambda$ is homogeneous. Thus 
\[\langle vX^2,Y^2+\lambda^2X^2\rangle = \langle vX^2,Y^2+u^{|Y|-|X|}X^2\rangle.\]
If $p'\neq 0$, we have $p'= u^{(|Y|-|X|-3)/2}$ and thus 
\begin{align*}
\langle vX^2,Y^2+\lambda^2X^2\rangle &= \langle vX^2,Y^2+w^2u^{|Y|-|X|-3}X^2\rangle\\
&=\langle vX^2,Y^2+(u^3+v^2+vw)u^{|Y|-|X|-3}X^2\rangle\\
&=\langle vX^2,Y^2+u^{|Y|-|X|}X^2\rangle.
\end{align*}
This proves the first part of the lemma. For the second part observe that
\[\langle v(Y+\lambda X)^2,X^2\rangle = \langle vY^2+v\lambda^2X^2,X^2\rangle = \langle vY^2,X^2\rangle.\qedhere\]
\end{proof}

As a consequence we obtain the following.

\begin{lemma}\label{lem:idealsvx2y2}
Let $I\subset H^*(BA_4)$ be a Steenrod closed parameter ideal with parameters $\{X,Y\}$ such that
$|X|<|Y|$. 
Then the set of all ideals of the form $\langle vX'^2,Y'^2\rangle$, where $\{X',Y'\}$ is any system of parameters for $I$,
consists of at most three elements, namely the elements $\langle vX^2,Y^2\rangle$, $\langle vX^2,Y^2+u^{|Y|-|X|}X^2\rangle$, $\langle X^2,vY^2\rangle$.
\end{lemma}

Now we prove a technical lemma.

\begin{lemma}\label{lem:vx2andy2plususx2}
Let $X,Y\in H^*(BA_4)$ be two homogeneous elements such that $\langle X,Y\rangle$ is a Steenrod closed parameter ideal and $s\coloneqq |Y|-|X|$ is nonnegative. Let $\alpha,\beta,\gamma,\delta\in H^*(BA_4)$ be elements satisfying 
\begin{align*}
\Sq(X) &=\alpha X+\beta Y,\\
\Sq(Y) &=\gamma X +\delta Y.
\end{align*}
Then $\langle vX^2,Y^2+u^{s}X^2\rangle$ is Steenrod closed if and only if
\[\gamma^2+(u+v+u^2)^s\alpha^2+u^s\delta^2+(u+v+u^2)^su^s\beta^2\in \langle v,Y^2+u^sX^2\rangle.\]
\end{lemma}
\begin{proof}
We know from \cref{lem:EvenEvenCase} that $\langle X^2,Y^2\rangle=\langle X^2,Y^2+u^sX^2\rangle$ is a Steenrod closed parameter ideal. In fact, we have 
\begin{align*}
\Sq(Y^2+u^sX^2)=&\Sq(Y)^2+(u+v+u^2)^s \Sq(X)^2\\
=&(\gamma^2X^2+\delta^2Y^2)+(u+v+u^2)^s(\alpha^2X^2+\beta^2Y^2)\\
=&(\gamma^2+(u+v+u^2)^s\alpha^2)X^2+(\delta^2+(u+v+u^2)^s\beta^2)Y^2\\
=&(\gamma^2+(u+v+u^2)^s\alpha^2+u^s\delta^2+(u+v+u^2)^su^s\beta^2)X^2+\\
&(\delta^2+(u+v+u^2)^s\beta^2)(Y^2+u^sX^2).
\end{align*}
By \cref{lem:Condition}, the ideal $\langle vX^2,Y^2+u^sX^2\rangle$ is also Steenrod closed if and only if 
\[\gamma^2+(u+v+u^2)^s\alpha^2+u^s\delta^2+(u+v+u^2)^su^s\beta^2\in \langle v,Y^2+u^sX^2\rangle.\qedhere\] 
\end{proof}
To apply \cref{lem:vx2andy2plususx2}, we need to know the coefficients $\alpha,\beta,\gamma,\delta$. In the following lemma, we compute these coefficients for the fibered case.
 
\begin{lemma}\label{lem:SqXSqYfibered}Let $t,c\geq 0$ with $c$ odd and $0\leq k \leq 2^t$. If $(X,Y)\coloneqq(v^k,u^{2^{t}c})$, then
\[
    \Sq(X)=\alpha X \mbox{ and }
    \Sq(Y)=\gamma X+\delta Y 
\]
with 
\begin{align*}
\alpha &=(1+u+v)^k,\\
\gamma&=\left(\sum_{j=0}^{c-1}\binom{c}{j} (u^{2^{t}}+u^{2^{t+1}})^jv^{2^{t}(c-j)-k}\right),\\
\delta&=(u^{2^{t}}+1)^c.
\end{align*}
Furthermore $\langle v^k,u^{2^{t}c}\rangle$ is 
a Steenrod closed parameter ideal if $k>0$.
\end{lemma}
\begin{proof}
We have 
\begin{align*}
\Sq(X)&=(v+uv+v^2)^k=(1+u+v)^kX,\\
 \Sq(Y)&= (u+v+u^2)^{2^{t}c}=(u^{2^{t}}+u^{2^{t+1}}+v^{2^{t}})^{c}\\
 &=\sum_{j=0}^c\binom{c}{j} (u^{2^{t}}+u^{2^{t+1}})^jv^{2^{t}(c-j)}\\ 
 &=\left(\sum_{j=0}^{c-1}\binom{c}{j} (u^{2^{t}}+u^{2^{t+1}})^jv^{2^{t}(c-j)-k}\right)X+(u^{2^{t}}+1)^cY.
\end{align*}
It is obvious that $v^k,u^{2^{t}c}\in \bbF_2[u,v]$ are coprime elements of positive degree for $k>0$. Hence the ideal $\langle v^k,u^{2^{t}c}\rangle$ is a Steenrod closed parameter ideal.
\end{proof}

In the following lemma we consider the ideals of the form
$\langle vX^2,Y^2\rangle$, $\langle vX^2,Y^2+u^{|Y|-|X|}X^2\rangle$, or  $\langle vY^2,X^2+u^{|X|-|Y|}Y^2\rangle$
as in \cref{lem:idealsvx2y2}
when $\langle X, Y \rangle$ is a fibered ideal with $(X, Y)=(v^k, u^l)$. To classify the fibered ideals in \cref{thm:classification_fibered}, 
we will use parts \eqref{lem:fibered_step1i} and \eqref{lem:fibered_step1ii} of \cref{lem:fibered_step1}. The third statement is used in the next section, in the proof of \cref{thm:classification_steenrodclosedparameterideals}. 

\begin{lemma}\label{lem:fibered_step1}
Let $(X,Y)\coloneqq(v^k,u^{2^{t}c})$, 
where $t, c \ge 0$ with $c$ odd, and $1\le k\le 2^{t}$. Then the following hold:
\begin{enumerate}
    \item \label{lem:fibered_step1i}$\langle vX^2,Y^2\rangle=\langle v^{2k+1},u^{2^{t+1}c}\rangle$ is a Steenrod closed parameter ideal if and only if $k<2^t$.
    \item \label{lem:fibered_step1ii} 
    If $|X|<|Y|$, then
    $\langle vX^2,Y^2+u^{|Y|-|X|}X^2\rangle$ is not a Steenrod closed parameter ideal.
    \item \label{lem:fibered_step1iii} If $|X|>|Y|$, then the ideal $\langle vY^2,X^2+u^{|X|-|Y|}Y^2\rangle$ 
    is a Steenrod closed parameter ideal if and only if $k=2^t$ and $c=1$.
\end{enumerate}
\end{lemma}

\begin{proof}
We prove \eqref{lem:fibered_step1i}. 
By \cref{lem:SqXSqYfibered}, if $k<2^t$, then
$\langle vX^2,Y^2\rangle=\langle v^{2k+1},u^{2^{t+1}c}\rangle$ is a Steenrod closed parameter ideal. If $k=2^t$ we want to use \cref{lem:Condition} to conclude that $\langle vX^2,Y^2\rangle$ is not Steenrod closed. Note that $\langle X^2,Y^2\rangle$ is a Steenrod closed parameter ideal by \cref{lem:EvenEvenCase} and thus we have to show that $\gamma ^2 \notin \langle v,u^{2^{t+1}c}\rangle$, where $\gamma$ is as in \cref{lem:SqXSqYfibered}. 

Modulo the ideal $\langle v, u^{2^{t+1}c} \rangle$, we have
\[\gamma \equiv \binom{c}{c-1}(u^{2^t}+u^{2^{t+1}})^{c-1}\equiv u^{2^t(c-1)}(1+u^{2t})^{c-1}.\]
Hence $\gamma ^2\equiv u^{2^{t+1} (c-1)} (1+u^{4t})^{c-1}$ mod $\langle v , u^{2^{t+1}} c \rangle$. The summand with the lowest exponent is $u^{2^{t+1}(c-1)}$ and this summand is nonzero modulo $\langle v, u^{2^{t+1}c}\rangle$. Hence $\gamma^2$ is not in $\langle v, u^{2^{t+1}c} \rangle$.

Now we want to prove \eqref{lem:fibered_step1ii}. Consider the ideal $I=\langle vX^2,Y^2+u^{s}X^2\rangle$ for $s=|Y|-|X|\geq 0$. 
Let  $\alpha,\gamma,\delta$ be as in \cref{lem:SqXSqYfibered}. 
By \cref{lem:vx2andy2plususx2}, the ideal $I$ is Steenrod closed if and only if
\[z\coloneqq\gamma^2+(u+v+u^2)^s\alpha^2+u^s\delta^2\]
is in $\langle v, Y^2+ u^s X^2\rangle$. Since $k\geq 1
$, we have $\langle v, Y^2+ u^s X^2\rangle=\langle v,u^{2^{t+1}c}\rangle$. Modulo this ideal, we have
\begin{align*}
&z \equiv (u^{2^{t+1}}+u^{2^{t+2}})^{c-1}v^{2^{t+1}-2k}+(u+u^2)^s(1+u)^{2k}+u^s(u^{2^{t+1}}+1)^c.
\end{align*}

First, consider the case $k<2^t$. Then $\gamma \equiv 0$ and we obtain 
\begin{align*}
    z\equiv (u+u^2)^s (1+u)^{2k} + u^s(u^{2^t+1}+1)^{c}&\equiv u^s((1+u)^{2k+s}+(u^{2^{t+1}}+1)^{c})\\
    &\equiv u^s((1+u)^{2^{t+1}c-k} +(1+u^{2^{t+1}})^c)
\end{align*}
since $s=|Y|-|X|=2^{t+1}c-3k$. Multiplying with the unit $(1+u)^{k}$ of $\bbF_2 [u]/\langle u^{2^{t+1}c}\rangle$ and writing $k = 2^l l'$ with $l'$ odd yields
\[u^s(1+u^{2^{t+1}})^c(1 +(1+u^{2^l})^{l'}).
\]
The summand with the lowest exponent is $u^s u^{2^l}=u^{s+2^l}$ and this summand is nonzero since the exponent $s+2^l\leq s+k = 2^{t+1}c-2k$ is smaller than $2^{t+1}c$.
Hence $z$ is not in the ideal $\langle v, u^{2^{t+1} c}\rangle.$

In the case $k=2^t$, we get $s=2^{t+1}c-3\cdot 2^t=(2c-3)2^t$. By assumption $s>0$ and hence  $c\geq 3$. Modulo the ideal $\langle v,u^{2^{t+1}c}\rangle$, we have
\begin{align*}
z  \equiv  & (u^{2^{t+1}}+u^{2^{t+2}})^{c-1}+(u+u^2)^s(1+u)^{2k}+u^s(u^{2^{t+1}}+1)^c \\
 \equiv & u^s \Bigl ( u^{2^t}  (1+u)^{2^{t+1} (c-1)}+(1+u)^{2^t (2c-1)}+(1+u)^{2^{t+1}c} \Bigr ).
\end{align*}
Multiplying with $(1+u)^{2^{t+1}}$ yields
\begin{align*}
(1+u) ^{2^{t+1}} z \equiv & u^s (1+u) ^{2^{t+1}c} \bigl ( u^{2^t}+(1+u)^{2^t}+(1+u)^{2^{t+1}} \bigr )\\
\equiv &  u^s (1+u)^{2^{t+1} c} u^{2^{t+1}} \equiv u^{(2c-1)2^t}  (1+u) ^{2^{t+1}c}. 
\end{align*}
The nonzero term of lowest degree in this polynomial $u^{(2c-1)2^t}$ does not lie in the ideal 
$\langle v, u^{2^{t+1}c}\rangle$. This gives that $(1+u) ^{2^{t+1}}z \not \equiv 0$ modulo $ \langle v, u^{2^{t+1} c} \rangle$.
From this we can conclude that 
$z \not \in \langle v, u^{2^{t+1}c} \rangle$, hence $I$ is not Steenrod closed.

Finally we want to prove \eqref{lem:fibered_step1iii}. Since $r\coloneqq |X|-|Y| =3 k -2^{t+1}c\ge 0$, we have $c=1$.
By \cref{lem:vx2andy2plususx2}, we have that $\langle vY^2,X^2+u^{|X|-|Y|}Y^2\rangle$ is Steenrod closed, if and only if 
\[z'\coloneqq 0^2+(u+v+u^2)^r\delta^2+u^r\alpha^2+(u+v+u^2)^ru^r\gamma^2\in \langle v,X^2+u^rY^2\rangle.\]
Modulo $\langle v,X^2+u^rY^2\rangle=\langle v,u^{|X|}\rangle=\langle v, u^{3k}\rangle$, we have
\begin{align*}
z'\equiv &(u+u^2)^r(u^{2^{t}}+1)^2+u^r(1+u)^{2k}+(u+u^2)^r u^r v^{2^{t+1}-2k} \\
=& u^r \Bigl ( (1+u) ^{r+2^{t+1}}+(1+u)^{2k} + u^r (1+u)^r v^{2^{t+1} -2k} \Bigr ). 
\end{align*}
If $k=2^t$, then we have $r=2^t$ and thus 
\begin{align*}
z'& \equiv u^{2^t} (1+u) ^{2^t} \Bigl ( (1+u) ^{ 2^{t+1} } + (1+u ) ^{2^t} +u^{2^t} \Bigr ) =(1+u)^{2^t}u^{3\cdot 2^t}\equiv 0.
\end{align*}
If $k<2^t$, then putting $r=3k-2^{t+1}$ gives
\[
   z' \equiv u^{3k-2^{t+1}} \Bigl ( (1+ u) ^{3k} + (1+u)^{2k} \Bigr )
     =  u^{3k-2^{t+1}} (1+u)^{2k} ((1+u)^k+1).
\]
Writing $k=2^l l'$ with $l'$ odd yields
\[ z'\equiv u^{3k-2^{t+1} } (1+u^{2^{l+1}})^{l'} \Bigl ( (1+u^{2^l})^{l'} +1 \Bigr ). \]
The summand with the lowest exponent is $u^{3k-2^{t+1}}u^{2^l}$ and this summand is nonzero since the exponent $3k-2^{t+1}+2^l$
is strictly less than $3k$.
\end{proof}

\begin{theorem}\label{thm:classification_fibered}
An ideal $I \subset H^* (BA_4)$ is fibered, if and only if $I=\langle v^k, u^l\rangle$ with $l=2^t c$ for $c$ odd and $1 \leq k\leq 2^t$. Furthermore, every fibered ideal can be written uniquely as above.
\end{theorem}
\begin{proof}
These ideals are Steenrod closed parameter ideals by \cref{lem:SqXSqYfibered}. Uniqueness follows from \cref{lem:lowergeneratorunique} 
for degree reasons.

We proceed by induction. Let $\langle v^k,Y\rangle$ be a fibered ideal. We assume that if $\langle v^{k'},Y'\rangle$ is fibered with $3k'+|Y'|<3k+|Y|$, then the ideal $\langle v^{k'},Y'\rangle$ has a system of parameters of the form $\{v^{k'},u^{l'}\}$ as above.

Next we want to show that we can assume that $\Sq^1(Y)=0$. If $3k<|Y|$, this follows from \cref{lem:sq1xeq0sq1yeq0} as $\Sq^1(v^k)=0$. If $|Y|<3k$ and $\Sq^1(Y)\neq 0$, then $v^k=\Sq^1(Y)$ by \cref{lem:Sq1Xneq0means} and thus the ideal is twisted. But the only twisted ideal, where one of the generators is of the form $v^i$ is $\langle u,v\rangle$. This is an ideal of the desired form. We thus can assume that we have $\Sq^1(Y)=0$ in general.

If $|Y|$ is odd, then $v$ divides $Y$ by \cref{lem:General}, but then $v^k$ and $Y$ are not coprime. So we can assume $|Y|$ is even. By \cref{lem:General}, this means that $Y=y^2$ for some $y$ in $H^*(BA_4)$.

If $k=2k'$, then we get by \cref{lem:EvenEvenCase} that
$\langle v^{k'},y\rangle$ is also a Steenrod closed parameter ideal and hence $\langle v^{k'},y\rangle =\langle v^{k'},u^{l'}\rangle$ by induction assumption. Thus 
\[\langle v^k,Y\rangle = \langle v^{k'},y\rangle^{[2]}=\langle v^{k'},u^{l'}\rangle^{[2]}=\langle v^{2k'},u^{2l'}\rangle.\]
It remains to look at the case $k=2k'+1$. If $k'=0$, we have $v^k=v$. Let us write $Y=p+wp'$ with $p,p'\in \bbF_2[u,v]\subset H^*(BA_4)$. We may leave out all monomials divisible by $v$ and thus we can assume that $p,p'\in \bbF_2[u]$. Since $Y$ has even degree we get $p'=0$. The polynomial $p$ cannot be zero, since otherwise $\langle v,Y\rangle =\langle v\rangle$ and this is not a parameter ideal. Hence $\langle v, Y\rangle =\langle v,u^{|Y|/2}\rangle $.

If $k'>0$, then by \cref{lem:EvenOddCase}, we have that
$\langle v^{k'},y\rangle$ is also a Steenrod closed parameter ideal and hence by induction assumption $\langle v^{k'}, y \rangle = \langle v^{k'},u^{l'}\rangle$
for some $k'\le 2^{t'}$ where $2^{t'}$ is the largest power of $2$ dividing $l'$.
Since $y\in \langle v^{k'},u^{l'}\rangle$, there exist homogeneous elements $\mu,\lambda$ such that $y=\mu u^{l'}+\lambda v^{k'}$. 
By \cref{lem:lowergeneratorunique}, the degrees of parameters are unique, hence we must have $|y|=|u^{l'}|$ and thus $\mu\in \bbF_2$. Since parameters are coprime, it follows that $\mu=1$ and $y=u^{l'}+\lambda v^{k'}$. If $\lambda=0$, then $\langle v^k,Y\rangle=\langle v^{2k'+1}, y^2\rangle= \langle v^{2k'+1},u^{2l'}\rangle$. If $\lambda \neq 0$, then by \cref{lem:vX2Y2normalform}, we get that $\langle v^k,Y\rangle=\langle v^{2k'+1},(u^{l'}+\lambda v^{k'})^2\rangle$ is one of the ideals
$\langle v^{k},u^{2l'}\rangle$ or $\langle v^{k},u^{2l'}+u^{2l'-3k'}v^{2k'}\rangle$. Since the latter is not Steenrod closed by \cref{lem:fibered_step1}\eqref{lem:fibered_step1ii}, we have $\langle v^k,Y\rangle=\langle v^{2k'+1},u^{2l'}\rangle$. By assumption, this ideal is a Steenrod closed parameter ideal and hence by \cref{lem:fibered_step1}\eqref{lem:fibered_step1i}, we get $k'<2^{t'}$. Thus $k=2k'+1\le 2^{t'+1}$ and $2^{t'+1}$ is the largest power of $2$ dividing $l=2l'$.
\end{proof}

\section{Classification of Steenrod closed parameter ideals}\label{sect:MixedCase}
 
In Sections \ref{sect:Twisted} and \ref{sect:FiberedCase} we classified the twisted and fibered Steenrod closed parameter ideals, respectively. In this section, we will classify all Steenrod closed parameter ideals.

In \cref{sect:Twisted} we have shown that every twisted ideal $I$ is equal to $\langle x_n, y_n\rangle$ for some $n\geq 1$, where the sequence of pairs 
$(x_n, y_n)$ is defined recursively by $(x_1, y_1)=(u,v)$ and $x_n=ux_{n-1} ^2 +y_{n-1}^2$, $y_n=vx_{n-1}^2$ for $n\geq 2$.
 Starting with such a pair $(x,y)$, we obtain new pairs using some transformations. One of the transformations we use is 
$$D: (X, Y)\mapsto (X^2, Y^2)$$ 
defined by squaring both coordinates.
We proved in \cref{lem:EvenEvenCase} that a pair $(X,Y)$ generates a Steenrod closed parameter ideal if and only if $D(X,Y)$ does so.
Another transformation is a generalization of the recursion we defined above. It only applies to the images of the recursive application of the transformation $D$ to $(x,y)$.
For each $m \geq 1$, we define
$$S \colon (x^{2^m}, y^{2^m}) \mapsto ( vx^{2^m}, u^{2^{m-1}} x^{2^m}+y^{2^m}).$$
Note that $S (x^2, y^2) = (vx^2, ux^2+y^2)$. Hence $S(x^2,y^2)$ generates the next twisted ideal.

Finally we define the transformation 
\[T \colon (X, Y) \mapsto (vX, Y).\] 
By \cref{lem:EvenOddCase}, if $(v X^2, Y^2)$ generates a Steenrod closed parameter ideal then so does $(X, Y)$. For small values of $m$, we have the following table.
\[
\xymatrix @-0.9pc{
(x,y) \ar[r]^-D & (x^2, y^2) \ar[d]^{S} \ar[r]^-D & (x^4, y^4) \ar[d]^{S} 
\ar[r]^-D & (x^8, y^8) \ar[d]^{S} \ar[r]^-D & \ldots\\
& (vx^2, ux^2+y^2) \ar[rd]^D &  ( vx^4, u^2 x^4+y^4)  \ar[rd]^D\ar[d]^T & (vx^8, u^4 x^8 +y^8) \ar[d]^T &  \\
& & (v^2 x^4, u^2x^4 +y^4) \ar[rdd]^D  & (v^2 x^8, u^4x^8+y^8) \ar[d]^T& \\
& & & (v^3 x^8, u^4 x^8 +y^8) \ar[d]^T & \\
& & & ( v^4 x^8, u^4 x^8 + y^8) & }
\]

Note that all the pairs in the above diagram have a specific form. We give a specific name for all the pairs obtained this way.

\begin{definition}\label{def:mixed}
A pair of homogeneous classes $(X,Y)$ in $H^* (BA_4)$ is called a \emph{mixed pair} if $$(X, Y) \coloneqq (v^ix_n^{2^m},u^{2^{m-1}}x_n^{2^{m}}+\Sq^1(x_n)^{2^{m}}),$$
for some $n\geq 1$, $m \geq 1$, and $0\le i \leq 2^{m-1}$,
where $x_n$ and $\mu_n$ are as in \cref{lem:SeqAreGood} satisfying that $\langle x_n,\Sq^1(x_n)\rangle$ is a Steenrod closed parameter ideal.
\end{definition}
\begin{remark}\label{eq:formulaformixedY} 
For $n \geq 1$ we have
$$
u^{2^{m-1}}x_n^{2^{m}}+\Sq^1(x_n)^{2^{m}} = (ux_n^2+y_n^2)^{2^{m-1}} = x_{n+1}^{2^{m-1}}.
$$
So, a mixed pair can also be expressed as a pair $( v^i x_n ^{2^m} , x_{n+1} ^{2^{m-1}})$ as it is done in \cref{thm:Intro-classification_steenrodclosedparameterideals}. Also note that when $i=2^{m-1}$, the mixed pair becomes the $(y_{n+1} ^{2^{m-1}}, x_{n+1} ^{2^{m-1}})$ pair. 
In particular, $(vx_n ^2, ux_n^2+y_n ^2)=(y_{n+1}, x_{n+1})$, so in this case the ideal generated by a mixed pair is a twisted ideal. We exclude the case $i=2^{m-1}$ from the subsequent statements about mixed pairs to avoid overlaps.
\end{remark}

Another overlap occurs when $n=1$ and $i=0$. In this case the ideal generated by a mixed pair is equal to $\langle v^{2^m}, u^{2^m} \rangle $ which is a fibered ideal. We exclude this case in our definition of mixed ideals.
 
\begin{definition}\label{def:mixedideal}
For $m \geq 1$, the ideal generated by a mixed pair
$(v^ix_n^{2^m},x_{n+1}^{2^{m-1}})$, with either
$n=1$ and $1\le i <2^{m-1}$, or $n \ge 2$ and $0\le i<2^{m-1}$, is called a \emph{mixed ideal}.
\end{definition}

The main aim of this section is to prove that every mixed ideal is a Steenrod closed parameter ideal
and all the Steenrod closed parameter ideals which are not fibered or twisted are mixed ideals. 
We start with proving one direction of our claim.

\begin{lemma}\label{lem:SqXSqYmixed}
Let 
$(X,Y)\coloneqq (v^ix_n^{2^m},u^{2^{m-1}}x_n^{2^{m}}+\Sq^1(x_n)^{2^{m}})$ be a mixed pair with $n\geq 1$, $m\ge 1$, and $0\le i <2^{m-1}$. Then we have
\[
\Sq(X)=\alpha X+\beta Y\mbox{ and }
\Sq(Y)=\gamma X+ \delta Y,
\]
with 
\begin{align*}
\alpha&=(1+u+v)^i(\mu_{n-1}^{2^m}+x_n^{2^m}+\mu_{n-1}^{2^m}u^{2^{m-1}}), \\
\beta &=(v+uv+v^2)^i\mu_{n-1}^{2^m}, \\ 
\gamma &=\mu_{n}^{2^{m-1}} v^{2^{m-1}-i}, \text{  and   }\ \delta =\mu_{n}^{2^{m-1}}+Y.
\end{align*}
Furthermore, $X$ and $Y$ are coprime and hence $\langle X,Y \rangle$ is a Steenrod closed parameter ideal.
\end{lemma}
\begin{proof}
By \cref{lem:SeqAreGood}, we have
\begin{align*} 
\Sq(x_n)& =(\mu _{n-1} +x_n) x_n+\mu_{n-1} y_n \\
\Sq(y_n)&=v\mu_{n-1} x_n +((u+1)\mu_{n-1} +x_n +y_n )y_n.
\end{align*}
This gives
\begin{align*}
\Sq(X) &= \Sq(v)^i\Sq(x_n)^{2^m}\\
&=(v+uv+v^2)^i\bigl ((\mu_{n-1}^{2^m}+x_n^{2^m})x_n^{2^m}+\mu_{n-1}^{2^m} y_n^{2^m} \bigl )\\
&=(v+uv+v^2)^i\Bigl ((\mu_{n-1} ^{2^m}+x_n^{2^m})x_n^{2^m}+\mu_{n-1}^{2^m}(u^{2^{m-1}}x_n^{2^m}+Y)\Bigr )\\
&=(1+u+v)^i(\mu_{n-1} ^{2^m}+x_n ^{2^m}+\mu_{n-1} ^{2^m} u^{2^{m-1}})X+(v+uv+v^2)^i\mu_{n-1}^{2^m}Y.
\end{align*}
Hence the formula for $\Sq(X)$ given in the lemma holds. To verify the formula for $\Sq(Y)$, observe that
\begin{equation}\label{eq:formulaformixed}
v^{2^{m-1}-i}X = v^{2^{m-1}}x_n^{2^m}=y_{n+1}^{2^{m-1}} \mbox{ and }Y=x_{n+1}^{2^{m-1}}\end{equation} by \cref{eq:formulaformixedY}.
So we have:
\begin{align*}
\Sq(Y)&=\Sq(x_{n+1})^{2^{m-1}}=(\mu_{n}^{2^{m-1}}+x_{n+1}^{2^{m-1}})x_{n+1}^{2^{m-1}}+\mu_{n}^{2^{m-1}} y_{n+1}^{2^{m-1}}\\
&=(\mu_{n}^{2^{m-1}}+Y)Y+\mu_{n}^{2^{m-1}} v^{2^{m-1}-i}X.
\end{align*}

Since $n\ge 1$, the generators $X,Y$ have degrees bigger than $0$. It remains to show that $X$ and $Y$ are coprime. Any common divisor would either be a divisor of $x_n$, in which case it would also divide $\Sq^1(x_n)$, which contradicts the assumption that $x_n$ and $\Sq^1(x_n)$ are coprime; or $v$ would be a common divisor. Since $v$ divides $\Sq^1(x_n)$ by \cref{lem:xymumodv}, it would also divide $x_n$ which contradicts again that $x_n$ and $\Sq^1(x_n)$ are coprime.
\end{proof}

Thus we have shown that any mixed ideal as defined in \cref{def:mixedideal} is a Steenrod closed parameter ideal.
Using the formulas in \cref{lem:SqXSqYmixed}, we immediately obtain the following:

\begin{lemma}\label{lem:SqXSqYmixedCoeffmodv}
Let 
$(X,Y)\coloneqq (v^ix_n^{2^m},u^{2^{m-1}}x_n^{2^{m}}+\Sq^1(x_n)^{2^{m}})$ be a mixed pair with $n\geq 1$, $m\ge 1$, and $0\le i <2^{m-1}$.
Modulo $v$ we have $X \equiv 0$ for $i>0$ and $X\equiv u^{2^{n+m}-2^m}$ for $i=0$, and $Y \equiv u^{2^{n+m}-2^{m-1}}$. Moreover, the residue classes of the coefficients from \cref{lem:SqXSqYmixed} 
are given by  
\begin{align*}
\alpha \equiv & (1+u)^i(u^{2^{n+m}-2^{m-1}}+1)/(u^{2^{m-1}}+1),\\
\beta \equiv&\begin{cases} (u^{2^{n+m}-2^{m}}+1)/(u^{2^{m}}+1)&i=0\\ 0&i>0
\end{cases},\\
\gamma \equiv &0, \text{  and  }\ \delta \equiv (u^{2^{n+m}}+1)/(u^{2^{m-1}}+1).
\end{align*}
\end{lemma}

\begin{proof}
 By \cref{lem:xymumodv}, we have  
\[x_n \equiv u^{2^{n}-1}, \ y_n\equiv 0,\mbox{ and } \mu_n\equiv (u^{2^{n+1}-1}+1)/(u+1)\]
and thus $Y \equiv u^{2^{m-1}} (u^{2^{n}-1})^{2^{m}}\equiv 
 u^{2^{n+m}-2^{m-1}}$ and $X \equiv 0$ for $i>0$ and $X\equiv u^{2^{n+m}-2^m}$ for $i=0$. 
Now the result follows from 
plugging these into the formulas in \cref{lem:SqXSqYmixed} and reducing everything mod $v$.
\end{proof}

We have the following lemma for mixed pairs. 

\begin{lemma}\label{lem:nonfibered_step1}
Let 
$(X,Y)\coloneqq (v^ix_n^{2^m},u^{2^{m-1}}x_n^{2^{m}}+\Sq^1(x_n)^{2^{m}})$ be a mixed pair with $n\geq 1$, $m\ge 1$, and $0\le i <2^{m-1}$.
Then the following hold:
\begin{enumerate}
    \item $\langle vX^2,Y^2\rangle$ is a Steenrod closed parameter ideal.
    \item If $|X|<|Y|$, then $\langle vX^2,Y^2+u^{|Y|-|X|}X^2\rangle$ is not a Steenrod closed parameter ideal.
\end{enumerate}
\end{lemma}

\begin{proof}
 Note that the ideal $$\langle vX^2,Y^2 \rangle =\langle v^{2i+1}x_n^{2^{m+1}},u^{2^{m}}x_n^{2^{m+1}}+\Sq^1(x_n)^{2^{m+1}} \rangle$$ is also of the form in \cref{lem:SqXSqYmixed}. 
Since $0 \leq i< 2^{m-1}$ and $m\geq 1$, we have $2i+1 < 2^m$. Thus by \cref{lem:SqXSqYmixed}, we can conclude that   
$\langle vX^2,Y^2\rangle$ is a Steenrod closed parameter ideal. 

 Now consider $I=\langle vX^2,Y^2+u^sX^2\rangle$ with $s=|Y|-|X|$. Write $\Sq(X)=\alpha X + \beta Y$, $\Sq(Y)=\gamma X + \delta Y$. By
 \cref{lem:vx2andy2plususx2}, the ideal $I$ is Steenrod closed if and only if the coefficient
\[z\coloneqq\gamma^2+(u+v+u^2)^s\alpha^2+u^s\delta^2+(u+v+u^2)^su^s\beta^2\]
lies in the ideal $\langle v,Y^2+u^sX^2\rangle$. By \cref{lem:SqXSqYmixedCoeffmodv}, this ideal simplifies to
\[\langle v, Y^2+u^sX^2\rangle = \begin{cases}
 \langle v\rangle &i=0\\
\langle v,u^{|Y|}\rangle&i>0.\end{cases}
\]
 It is helpful to record the degrees of $X,Y$ using that $|x_n|=2^{n+1}-2$.
\[
|X|= 3i+2^{m+n+1}-2^{m+1},\quad
|Y|= 2^{n+m+1}-2^m,\quad
s=|Y|-|X|=2^m-3i.
\]
First, we consider the case $i>0$. We have mod $\langle v,u^{|Y|}\rangle$:  
$$z\equiv (u+u^2)^s(1+u)^{2i}(u^{2^{n+m+1}-2^{m}}+1)/(u^{2^{m}}+1)+u^s(u^{2^{n+m+1}}+1)/(u^{2^{m}}+1).$$
It suffices show that the coefficient multiplied with the unit $1+u^{2^m}$ in $\bbF_2[u]/\langle u^{|Y|}\rangle$ is nonzero. We have
$$(1+u^{2^m})z  \equiv u^s ( (1+u)^{2i+s}(u^{2^{n+m+1}-2^{m}}+1)+(u^{2^{n+m+1}}+1) )  \equiv u^s((1+u)^{2^{m}-i}+1)$$
since $u^{2^{n+m+1}}\equiv u^{2^{n+m+1}-2^m}\equiv 0$ mod $u^{|Y|}$.
Now write $i=l'2^l$ with $l'$ odd and thus $l<m-1$. Then $$(1+u)^{2^{m}-i}=(1+u^{2^l})^{2^{m-l}-l'}= 1 + u^{2^l} + \cdots$$ since $2^{m-l}-l'$ is odd and it suffices to show that $u^{s+2^l}$ is nonzero in $\bbF_2[u]/\langle u^{|Y|}\rangle$. Note that $|X|-i = 2i+2^{m+n+1}-2^{m+1}>0$. We thus have $s+2^l\leq s+i<s+|X|=|Y|$ and hence the coefficient is indeed nonzero in $\bbF_2[u]/\langle u^{|Y|}\rangle$.
Hence $z$ does not lie in the ideal $\langle v,Y^2+u^sX^2\rangle$

Secondly, consider the case $i=0$. Then by \cref{lem:SqXSqYmixedCoeffmodv} we have that $Y\equiv u^{2^{m+n}-2^{m-1}}$ and $X\equiv u^{2^{m+n}-2^{m}}$. Thus 
$Y^2+u^sX^2 \equiv 0$ mod $v$ and hence the ideal $\langle vX^2,Y^2+u^sX^2\rangle$ cannot be a parameter ideal, since both generators are divisible by $v$.  
\end{proof}

With swapped roles for $Y$ and $X$, we do not obtain any Steenrod closed parameter ideals.

\begin{lemma}\label{lem:nonfibered_step2} 
Let 
$(X,Y)\coloneqq (v^ix_n^{2^m},u^{2^{m-1}}x_n^{2^{m}}+\Sq^1(x_n)^{2^{m}})$ be a mixed pair with $n\geq 1$, $m\ge 1$, and $0\le i <2^{m-1}$.
Then the following hold:
\begin{enumerate}
    \item The ideal $\langle vY^2,X^2\rangle$ is not a Steenrod closed parameter ideal.
    \item If $|X|>|Y|$, then the ideal $\langle vY^2,X^2+u^{|X|-|Y|}Y^2\rangle$ is not a Steenrod closed parameter ideal.
\end{enumerate}
\end{lemma}
\begin{proof}
Consider $I=\langle vY^2,X^2\rangle$. If $i>0$, then both generators are divisible by $v$ and hence the ideal cannot be a parameter ideal. If $i=0$, we have $\langle vY^2,X^2\rangle = \langle v\Sq^1(x_n)^{2^{m+1}}, x_n^{2^{m+1}}\rangle$ and this ideal is not Steenrod closed, because it does not contain $\Sq^{2^{m+1}}(X^2)=\Sq^1(x_n)^{2^{m+1}}$. Otherwise it would be a multiple of $x_n^{2^{m+1}}$ for degree reasons. This contradicts the assumption that $\langle x_n,\Sq^1(x_n)\rangle$ is a parameter ideal.

 Now consider the case of $I=\langle vY^2,X^2+u^sY^2\rangle$ with $s=|X|-|Y|$. We are applying \cref{lem:vx2andy2plususx2} to the pair $Y,X$. Thus $X$ in \cref{lem:vx2andy2plususx2} corresponds to $Y$ here. Write $\Sq(Y)= \delta Y+\gamma X$, $\Sq(X)=\beta Y+\alpha X$ with the coefficients from \cref{lem:SqXSqYmixed}. Then the ideal $I$ is Steenrod closed if and only if
 \[z\coloneqq\beta^2+(u+v+u^2)^s \delta^2 +u^s\alpha^2+ (u+v+u^2)^s u^s \gamma^2
 \]
 lies in $\langle v, X^2 + u^s Y^2\rangle$. We will show that this is not the case. Since $s> 0$, we have $i>0$ and it follows from
 \cref{lem:SqXSqYmixedCoeffmodv} that
 \[\langle v, X^2 + u^s Y^2\rangle = \langle v, u^{|X|}\rangle.
 \]
It is helpful to record the degrees of $X,Y$ using that $|x_n|=2^{n+1}-2$:
\[|X|= 3i+2^{m+n+1}-2^{m+1},\quad
|Y|= 2^{n+m+1}-2^m,\quad 
s=|X|-|Y|= 3i -2^m.\]
We have mod $\langle v, u^{|X|}\rangle$:
$$z\equiv (u+u^2)^s(u^{2^{n+m+1}}+1)/(u^{2^{m}}+1) + u^s(1+u)^{2i}(u^{2^{n+m+1}-2^{m}}+1)/(u^{2^{m}}+1).$$
It suffices to show that this element is nonzero in $\bbF_2[u]/\langle u^{|X|}\rangle$. Multiplication with the unit $u^{2^{m}}+1$ in $\bbF_2[u]/\langle u^{|X|}\rangle$ yields
\begin{align*}
  (u^{2^m} +1) z \equiv  &  (u+u^2)^s(u^{2^{n+m+1}}+1)+ u^s(1+u)^{2i}(u^{2^{n+m+1}-2^{m}}+1)\\
    \equiv&u^s((1+u)^{2i}((1+u)^s+u^{2^{n+m+1}-2^{m}}+1))
\end{align*}
since $u^s u^{2^{n+m+1}} \equiv 0$ mod $u^{|X|}$.
Write $s= l' 2^l$ with $l'$ odd. Then $(1+u)^s= (1+u^{2^l})^{l'}$ and from $2i= 2(s+2^m)/3=2^{l+1}(l' +2^{m-l})/3$ it follows that $(1+u)^{2i}= (1+u^{2^{l+1}})^{(l'+2^{m-l})/3}$. We conclude that the smallest power of $u$ that appears in the sum is $u^s\cdot u^{2^l}$. This summand is nonzero in $\bbF_2[u]/\langle u^{|X|}\rangle$ since 
$2^l \leq s =3i-2^m < 3 \cdot 2^{m-1} -2^m \le 2^{m+n+1}-2^{m}=|Y|$ and hence $|u^s \cdot u^{2^l}| < |Y|+s=|X|$.
\end{proof}

We need the following lemma in the proof of our main theorem.

\begin{lemma}\label{lem:twisted_step}
Let $\langle x,\Sq^1(x)\rangle$ be a twisted Steenrod closed parameter ideal. Then the ideals $\langle vx^2,\Sq^1(x) ^2\rangle$ and $\langle x^2,v\Sq^1(x) ^2\rangle$ are not Steenrod closed parameter ideals.
\end{lemma}
 
\begin{proof}
The ideal $\langle x^2,v\Sq^1(x)^2 \rangle$ is not Steenrod closed: if $\Sq^2(x^2)=\Sq^1(x)^2$ was in $\langle x^2,v\Sq^1(x)^2 \rangle$, then it would be a multiple of $x^2$ for degree reasons. We would thus have $\Sq^1(x)^2=\lambda x^2$. This contradicts the assumption that $x$ and $\Sq^1(x)$ are coprime. The ideal $\langle vx ^2, \Sq (x )^2 \rangle$ is not a parameter ideal, because $vx^2$ and $\Sq^1(x)^2$ are not coprime as $v$ divides $\Sq^1(x)$ by \cref{lem:xymumodv} and \cref{thm:classification_twisted_pairs}.
\end{proof}

The main result of this section is the following classification theorem.
 
\begin{theorem}\label{thm:classification_steenrodclosedparameterideals} 
The set of Steenrod closed parameter ideals in $H^*(BA_4;\bbF_2)$ consists of 
\begin{enumerate}
\item\label{item:fibered} the \emph{fibered ideals} $\langle v^k,u^l\rangle$ with $l\ge 1$ and $1\le k\le 2^t$, where $2^t$ is the largest power of $2$ dividing $l$;
\item\label{item:twisted} the \emph{twisted ideals} $\langle x_n,\Sq^1(x_n)\rangle$ for $n\ge 2$, where $x_n$ is recursively defined as 
$x_1=u$ and $x_{n+1} = ux_n^2+\Sq^1(x_n)^2$;
\item\label{item:mixed} and \emph{the mixed ideals} $\langle v^ix_n^{2^m},x_{n+1}^{2^{m-1}}\rangle$, where $m\ge 1$, and either
$n=1$ and $1\le i <2^{m-1}$, or $n \ge 2$ and $0\le i<2^{m-1}$.
\end{enumerate}
\end{theorem}
 
\begin{proof}
We have classified the twisted ideals in \cref{thm:classification_twisted_pairs} and the fibered ideals in \cref{thm:classification_fibered}. The ideals in \eqref{item:mixed} are indeed Steenrod closed parameter ideals by \cref{lem:SqXSqYmixed} and \cref{eq:formulaformixedY}. 

We will show that any Steenrod closed parameter ideal that is neither twisted nor fibered is of the form \eqref{item:mixed}. 
We proceed by induction on the sum of the degrees of the parameters.
Let $I$ be a Steenrod closed parameter ideal that is neither fibered nor twisted. 

Let us first consider the case where both parameters have even degrees.
By \cref{pro:square_of_good_generators}, $I=J^{[2]}$ for some Steenrod closed parameter ideal $J$. By induction assumption $J$ is either fibered, twisted, or mixed. If $J$ is fibered, so is $I$. If $J$ is mixed, so is $I$. If $J$ is twisted, say $J=\langle x_n,\Sq^1(x_n)\rangle$ for $n\ge 2$, then $I= \langle x_n^2,ux_n^2+\Sq^1(x_n)^2\rangle =\langle x_n^2, x_{n+1})$ and thus it is a mixed  ideal with $i=0$ and $m=1$.

It remains to consider the case where at least one of the parameters has odd degree. By \cref{pro:generatorsoddeven}, 
they cannot both be odd. If one of the parameters has degree $3$, then $I=\langle v,y^2\rangle$ by \cref{pro:generatorsoddeven} and thus it is fibered by \cref{def:fibered}.
Otherwise we can find parameters $\{vx^2,y^2\}$ of $I$ such that $J=\langle x,y\rangle$ is also a Steenrod closed parameter ideal by \cref{pro:generatorsoddeven}. By induction the ideal $J$ is either fibered, twisted, or mixed.

First, consider the case that $J$ is fibered and thus generated by a pair of the form $(X,Y)=(v^k,u^{2^{t}c})$ with $1\leq k\leq 2^t$ and $c$ odd. The ideal $I$ has to be one of the ideals from \cref{lem:idealsvx2y2}:   
\begin{enumerate}[label=\roman*),leftmargin=*]
\item $\langle vX^2,Y^2\rangle=\langle v^{2k+1},u^{2^{t+1}c}\rangle$ is a Steenrod closed parameter ideal if and only if $k<2^t$ by \cref{lem:fibered_step1}\eqref{lem:fibered_step1i}. In this case, it is again fibered.
\item   If $|X|< |Y|$, then the ideal $I$ can be $\langle vX^2,Y^2+u^{|Y|-|X|}X^2\rangle$, but  
such ideals are not a Steenrod closed parameter ideals by \cref{lem:fibered_step1}\eqref{lem:fibered_step1ii}.
\item $\langle X^2,vY^2\rangle$ is not a Steenrod closed parameter ideal, since both generators are divisible by $v$.
\item If $|X|>|Y|$, we can have $I=\langle X^2+u^{|X|-|Y|}Y^2,vY^2\rangle$. By \cref{lem:fibered_step1}\eqref{lem:fibered_step1iii}, this ideal 
is a Steenrod closed parameter ideal if and only if $k=2^t$ and $c=1$. This gives $I=\langle v^{2^{t+1}}+u^{3\cdot 2^{t}},vu^{2^{t+1}}\rangle$ which is of the form \eqref{item:mixed} with $i=1$, $n=1$, and $m=t+1$.
\end{enumerate}

Secondly, consider the case that $J$ is twisted and nonfibered. This means that $J=\langle x_n, \Sq^1(x_n)\rangle$ for some $n\ge 2$. By \cref{rem:twistedUniqueParameters}, $\{x_n,\Sq^1(x_n)\}$ is the only choice of parameters for $J$. Neither of the ideals $\langle vx_n^2,\Sq^1(x_n)^2\rangle$ and $\langle x_n^2,v\Sq^1(x_n)^2\rangle$ is a Steenrod closed parameter ideal by \cref{lem:twisted_step}.
 
Thirdly, consider the case that $J$ is mixed, i.e., generated by a pair of the form 
\[(X,Y)=(v^ix_n^{2^m},u^{2^{m-1}}x_n^{2^{m}}+\Sq^1(x_n)^{2^{m}})\]
with $m\ge 1$, $n \ge 2$ and $0\le i<2^{m-1}$ or $m\ge 1$, $n=1$ and $1\le i <2^{m-1}$.
The ideal $I$ has to be one of the ideals from \cref{lem:idealsvx2y2}:
\begin{enumerate}[label=\roman*),leftmargin=*]
\item $\langle vX^2,Y^2\rangle$ is again generated by a pair of the form \eqref{item:mixed} by \cref{lem:nonfibered_step1}.
\item $\langle vX^2,Y^2+u^{|Y|-|X|}X^2\rangle$ if $|Y|>|X|$, which is not a Steenrod closed parameter ideal 
by \cref{lem:nonfibered_step1}. 
\item $\langle X^2,vY^2\rangle$ is not a Steenrod closed parameter ideal by \cref{lem:nonfibered_step2}.
\item $\langle X^2+u^{|X|-|Y|}Y^2,vY^2\rangle$ if $|X|>|Y|$, 
which is not a Steenrod closed parameter ideal by \cref{lem:nonfibered_step2}.
\end{enumerate}
This completes the proof of the theorem.
\end{proof}

To establish \cref{thm:Intro-classification_steenrodclosedparameterideals}, we also need the following result.

\begin{proposition}\label{pro:uniquenessindegrees}
For a given pair of natural numbers there is at most one Steenrod closed parameter ideal with parameters of these degrees. Additionally, the union of families listed in $\eqref{item:fibered}$-$\eqref{item:mixed}$ of \cref{thm:classification_steenrodclosedparameterideals} is disjoint.
\end{proposition}

\begin{proof}
For a Steenrod closed parameter ideal $I\subset H^*(BA_4)$, we will denote by $|I|$ the set of degrees $\{|X|,|Y|\}$, where $\{X,Y\}$ is a homogeneous system of parameters for $I$. We will show that the degrees of the Steenrod closed parameter ideals listed in \cref{thm:classification_steenrodclosedparameterideals} are all distinct.

There is no Steenrod closed parameter ideal $\langle X, Y\rangle$ such that $|X|$ and $|Y|$ are odd. Indeed, for twisted ideals this holds since the degrees of the parameters differ by $1$. For nontwisted ideals this follows from \cref{pro:generatorsoddeven}.

Note that if $I$ is a fibered ideal, then $|I|=\{3k, 2^{t+1} c\}$ with $c$ odd and $1 \leq k\leq 2^t$. If $I$ is twisted and not $\langle v, u\rangle$, then $|I|=\{2^{n+1}-2, 2^{n+1}-1\}$ for some $n\geq 2$.
Finally if $I$ is mixed then $|I|=\{3i+(2^r-1) 2^{m+1}, (2^{r+1} -1) 2^m \}$ for some $m\geq 1$ and $r\geq 1$,
where $0 \leq i < 2^{m-1} $ if $r \geq 2$ and $1 \leq i < 2^{m-1}$ if $r=1$. 

Suppose that there is a fibered and a twisted ideal different from $\langle v,u\rangle$ with same degree. Then we must have $3k=2^{n+1}-1$ and $2^{t+1} c= 2(2^n-1)$.
The second equality implies that $t=0$. Since $k \leq 2^t$, we have $k=1$. Then the first equality gives $n=1$ which is not possible since $n \geq 2$ by our assumption.

If there is a twisted and a mixed ideal with the same degree, then $2(2^n-1)=(2^{r+1} -1) 2^m$ and $2^{n+1} -1= 3i+(2^r-1)2^{m+1}$.
The first equality gives $m=1$. So we must have $i=0$ by the given inequalities. But then the second equality
does not hold since the right-hand side is even and the left-hand side is odd.

For the case of a fibered and a mixed ideal with the same degree, we need the following observations. If $I=\langle X,Y\rangle$ is a Steenrod closed parameter ideal with both parameters of even degrees, then there exists a unique Steenrod closed parameter ideal $J$ such that $I=J^{[2]}$ by \cref{pro:square_of_good_generators}. Note that $|J|=\{|X|/2,|Y|/2\}$. For $I$ fibered, the degree of $J$ is again the degree of a fibered ideal. For $I$ mixed, the degree of $J$ is either the degree of a mixed ideal or the degree of a twisted ideal. Hence by induction it suffices to prove the statement when one parameter is of odd degree and the other parameter has even degree.

Suppose that there is a fibered and a mixed ideal with the same degree. By the discussion above we can assume that one of the parameters is of odd degree. Then $k$ is odd, and we have $3k=3i+(2^r-1)2^{m+1}$ and $ 2^{t+1} c=(2^{r+1} -1) 2^m $. This implies that
$t+1=m$. The inequality $k \leq 2^t$ gives $k \leq 2^{m-1}$. From the first equality we obtain
$$ 3i + (2^r-1) 2^{m+1} \leq 3 \cdot 2^{m-1},$$
which gives $0 \leq 3i \leq (3-4 (2^r-1)) 2^{m-1}$. Since for $r \geq 1$, we have $3-4(2^r-1)<0$, this gives a contradiction.

Hence we can conclude that no two Steenrod closed parameter ideals of different type have parameters of the same degrees. In particular, the union in \cref{thm:classification_steenrodclosedparameterideals} is disjoint.

To complete the proof, we also need to eliminate the possibility that two different Steenrod closed parameter ideals of the same type have parameters with equal degrees.
This is clear in the twisted case since they are defined recursively. For the fibered case, assume that
there are two pairs of numbers $k, k', t, t', c, c'$ satisfying the conditions for the fibered ideal 
such that $\{3k, 2^{t+1}c\}=\{3k', 2^{t'+1} c'\}$. Since by induction we can assume that one of the parameters is of odd degree,
then we must have $3k=3k'$ since these two are the only possible odd numbers in the pair. As $c'$ and $c$ are odd, we obtain that $t=t'$ and $c=c'$, 
hence the two fibered ideals are equal. For the mixed case, suppose that 
for some $i, i', r, r', m, m'$ satisfying the conditions for a mixed ideal, we have
$$\{ 3i+(2^r-1) 2^{m+1}, (2^{r+1} -1) 2^m \}=\{ 3i'+(2{r'} -1) 2^{m'+1}, (2^{r'+1}-1) 2^{m'} \}$$
Again by induction we can assume that one of the degrees is odd. The only odd degree terms are the ones that involve $i$ and $i'$, so
we must have $(2^{r+1} -1)2^m= (2^{r'+1} -1) 2^{m'}$. This gives $r=r'$ and $m=m'$.
From this we obtain $i=i'$, and hence the two mixed ideals are equal.
\end{proof}
 
\cref{thm:classification_steenrodclosedparameterideals} and \cref{pro:uniquenessindegrees} together complete the proof of \cref{thm:Intro-classification_steenrodclosedparameterideals}, the main theorem of the paper.

\begin{remark}
From the statement of \cref{thm:Intro-classification_steenrodclosedparameterideals}, we can observe that parameters of a Steenrod closed parameter ideal can be chosen so that one of the parameters is divisible by $v$. In the fibered case this is clear from the definition.
In the twisted case it follows from the recursion formula for $x_n$. In the mixed case, this is obvious when $i >0$. If $i=0$, then
the mixed ideal is $\langle x_n ^{2^m} , x_{n+1} ^{2^m-1} \rangle=\langle x_n ^{2^m} , \Sq ^1 (x_n ) ^{2^m}\rangle$. So $v$ divides one of these parameters since $v$ divides one of the parameters in the twisted case.
\end{remark}

Calculating the degrees of the parameters and using uniqueness of \cref{pro:uniquenessindegrees} yields the following result. 

\begin{corollary}\label{cor:degreesofparameters} Let $I\subset H^*(BA_4)$ be a Steenrod closed parameter ideal. Then the (unordered) degrees of the parameters of $I$ are of the form
\begin{enumerate}
    \item \label{it:degreesofparametersi}$(3k,2l)$ with $l\ge 1$ and $1\le k\le 2^t$, where $2^t$ is the largest power of $2$ dividing $l$, or
\item \label{it:degreesofparametersii} $(3i+2^{m+n+1}-2^{m+1},2^{m+n+1}-2^m)$
for $m\ge 0$, $n \ge 1$ and $0\le i<2^{m-1}$.
\end{enumerate}
For each such pair there is a unique Steenrod closed parameter ideal with parameters of these degrees.
\end{corollary}

\begin{remark}\label{rem:oneDegreeDivisibleByThree} Note that at least one of the degrees of the parameters of a Steenrod closed parameter ideal is divisible by three by \cref{cor:degreesofparameters}. 

To decide whether a pair of natural numbers is in the list above, it helps to write the numbers in base two.
\end{remark}

\begin{lemma}\label{lem:rlogr}
    Let $c(r)$ denote the number of Steenrod closed parameter ideals with parameters of degree at most $r$. Then $c(r)\in O(r \log(r))$ .
\end{lemma}
\begin{proof}
Instead of estimating the number of pairs in \eqref{it:degreesofparametersi}, we relax the conditions and estimate the number of pairs $(k,l)$ with $1\leq l<r$ and $1\leq k\leq 2^t$, where $2^t$ is the largest power of $2$ dividing $l$. If $r=2^s$ is a power of $2$, then for fixed $0\leq t\leq s$, there are $2^{s-1-t}$ possibilities for $l$ as can be seen by writing $l$ in binary expansion. For fixed $l$, there are $2^t$ possibilities for $k$. Thus the number of pairs $(k,l)$ is given by
\[\sum_{t=0}^{s-1} 2^{s-1-t} \cdot 2^t = \sum_{t=0}^{s-1} 2^{s-1} = 2^{s-1}\cdot s.\]
For arbitrary $r$, let $s$ be the smallest integer greater or equal to $\log_2(r)$. We have $r\leq 2^s< 2r$. Then the number of pairs $(k,l)$ is bounded by 
$2^{s-1}\cdot s\leq r(\log_2(r)+1)$.

    For case \eqref{it:degreesofparametersii}, we estimate the number of pairs with $2^{m+n+1}-2^m\leq r$. Since $2^{m+n}<2^m(2^{n+1}-1)\le r$, it follows that $m+n\leq \log_2(r)$. Let $s$ be the biggest integer not greater than $\log_2(r)$. For a given pair $m,n$, there are $2^{m-1}$ choices for $i$ if $m>0$ and one choice if $m=0$. Thus the number of pairs is bounded by
    \begin{equation*}
    \sum_{n=1}^{s}(1+\sum_{m=1}^{s-n} 2^{m-1})
    = \sum_{n=1}^{s} 2^{s-n}
    \le 2^{s} \le r\,.
    \end{equation*}
Hence, we obtain that $c(r)\in O(r\log(r))$ as claimed.
\end{proof}

To obtain a classification of Steenrod closed parameter ideals in $H^* (B\SO(3))$, we need the following two results. The first one is an immediate consequence of the classification in \cref{thm:classification_steenrodclosedparameterideals}.

\begin{corollary}\label{cor:SteenrodclosedparameterIdeals_parametersinuv}
Every Steenrod closed parameter ideal of $H^*(BA_4)$ is of the form $\langle X,Y\rangle$ with $X,Y\in \bbF_2[u,v]\subset H^*(BA_4)$.
\end{corollary}

\begin{lemma}\label{lem:injectiononparameterideals}
The map sending a Steenrod closed parameter ideal in $H^*(B\SO(3))$ to its extension in $H^*(BA_4)$ is injective. \end{lemma}

\begin{proof}
Let $I$, $J$ be parameter ideals of $H^*(B\SO(3))$ with parameters $\{X,Y\}$ and $\{X',Y'\}$, respectively. Suppose that $RI=RJ$ for $R=H^*(BA_4)$. 

After renaming the parameters we may assume that $|X|<|Y|$ and $|X'|<|Y'|$. For degree reasons and since we work over $\bbF_2$, we have $X=X'$ and $Y'= Y+\lambda X$ for some homogeneous element $\lambda\in H^*(BA_4)$. We conclude from \cref{lem:divisionlandsininvariantring} that $\lambda\in \bbF_2[a,b]$ is $\GL_2(2)$-invariant. It follows that $I=J$.
\end{proof}

We obtain the following classification of Steenrod closed parameter ideals in $H^*(B\SO(3))$.
\begin{corollary}\label{cor:classification_SteenrodclosedParameterIdeal_SO3}
Restriction of ideals yields a bijection between Steenrod closed parameter ideals in $H^*(BA_4)$ and Steenrod closed parameter ideals in the Dickson algebra $H^*(B\SO(3))$. Thus the Steenrod closed parameter ideals in $H^*(B\SO(3))$ are the restrictions of the ideals listed in \cref{thm:classification_steenrodclosedparameterideals}.
\end{corollary}
\begin{proof}
 The extension from $H^*(B\SO(3))$ to $H^*(BA_4)$ defines an injection on Steenrod closed parameter ideals by \cref{lem:injectiononparameterideals} and \cref{lem:steenrodclosedinextension}. It is surjective by \cref{cor:SteenrodclosedparameterIdeals_parametersinuv}. The inverse of this bijection is restriction.
\end{proof}

\section{\texorpdfstring{The $k$-invariants of a free $G$-action}{The k-invariants of a free G-action}}
 
In \cite[Theorem~2]{Oliver}, Oliver showed that $A_4$ can not act freely on any finite CW-complex $X$ with $H^* (X; \bbZ ) \cong H^* (S^n \times S^n; \bbZ)$. The proof proceeds in two steps. First, $A_4$ can not 
act freely on any finite-dimensional $X$ with mod-$2$ cohomology of a finite product of $n$-spheres
\[H^* (X; \bbF_2) \cong H^* (\prod _k S^n ; \bbF_2)\] 
and $A_4$ acting trivially 
on cohomology (see \cite[Theorem~1]{Oliver}). Secondly, a Lefschetz fixed point argument ensures that $A_4$ acts trivially on cohomology provided that $H^* (X; \bbZ ) \cong H^* (S^n \times S^n; \bbZ)$ and $X$ is finite. We observe that this assumption can be weakened to mod-$2$ coefficients. 

\begin{theorem}\label{thm:nofreeactionequidim}
There is no finite, free $A_4$-CW-complex $X$ with $H^*(X;\bbF_2)\cong H^*(S^n\times S^n;\bbF_2)$. 
\end{theorem}
\begin{proof}
Suppose that $X$ is a finite, $A_4$-CW-complex with $H^*(X;\bbF_2)\cong H^*(S^n\times S^n;\bbF_2)$. By \cite[Theorem~1]{Oliver}, it suffices to show that if $A_4$ acts nontrivially on $H^{n}(X;\bbF_2)\cong \bbF_2\oplus \bbF_2$, then the $A_4$-action on $X$ is not free. Let $A_4\to \GL_2(2)\cong S_3$ be the representation of $A_4$ on $H^n(X;\bbF_2)$ and assume it is nontrivial. Since its kernel is a proper, normal subgroup of $A_4$, it follows that the kernel is $\bbZ/2\times\bbZ/2$ and that $C_3\subset A_4$ acts nontrivially on  $H^{n}(X;\bbF_2)$. Therefore, the trace of any generator of $C_3$ acting on $H^{n}(X;\bbF_2)$ is one. It follows from the Lefschetz fixed point theorem with mod-$2$ coefficients (see e.g. \cite[III.C Theorem~2]{brown1971}) that $C_3$ fixes a point of $X$. Hence $A_4$ does not act freely on $X$.
\end{proof}

In this section we consider free $A_4$-actions 
on a finite CW-complex $X$ which has the mod-$2$ cohomology of a product of two spheres $S^n \times S^m$ with $n\neq m$.
We first recall the general definitions and methods for studying free group actions on a finite-dimensional CW-complex. Let $G$ be a finite group and $X$ be a finite-dimensional $G$-CW-complex. The Borel construction for the $G$-CW-complex $X$ is a fibration 
$$X\rightarrow EG\times _G X \xrightarrow{\pi} BG$$ where $EG$ 
is the universal space for $G$, and $BG$ is the classifying space for $G$. 
Let $R$ be a commutative ring with unity. The
Serre spectral sequence associated to this fibration 
has $E_2$-page $$E_2^{s,t} =H^s (BG ; H^t (X; R))$$
and it converges to $H^* (EG\times_G X; R)$. 
The cohomology ring $H^*(EG\times _G X ; R)$ is called the 
equivariant Borel cohomology of $X$, denoted by $H^* _G (X; R)$.

If $G$ acts freely on $X$, then $EG \times _G X$ is homotopy equivalent 
to the orbit space $X/G$ which is a finite-dimensional CW-complex. This 
puts restrictions on the $E_{\infty}$-page of the Serre spectral sequence.
Since $H^n _G (X; R) \cong H^n (X/G; R)=0$ for $n>\dim X$,
we have $E_{\infty} ^{i,j}=0$ for $i+j > \dim X$.

The Serre spectral sequence is multiplicative and thus has a pairing  $$E^{s, t} _r \otimes E^{s', t'} _r \to E_r ^{s+s', t+t'} $$  
so that the differential $d_r$ is a derivation with respect to this product.  
 If for all $s,t \geq 0$, $H^s (BG; R)$ and $H^t (X; R)$ 
are free $R$-modules of finite type, and the $G$-action on the cohomology $H^*(X; R)$ is trivial, 
then 
$$E_2^{*,*} =H^* (BG ; H^* (X; R)) \cong H^* (X; R) \otimes H^* (BG; R)$$
as a bigraded algebra (see \cite[Proposition~5.6]{mccleary2001user}).
In fact, by the Universal coefficient theorem \cite[Theorem~5.5.10]{spanier1995algebraic} the isomorphism above holds  under the weaker assumption that $G$ acts trivially on $H^*(X; R)$ and
either $H^t(X;R)$ is a finitely generated, free $R$-module for each $t$, or $R$ is a field.
Note that here the isomorphism as bigraded algebras means that the product structure on $E_2 ^{*,*}$ 
coincides with the product induced by cup products on $H^* (X; R)$ and $H^* (BG; R)$. 
 
The product structure on the Serre spectral sequence gives an $E_r ^{*, 0}$-module
structure on $E_r ^{*, *}$. There is also a $H^* (BG; R)$-module
structure on $E_r ^{*, *}$ induced by the constant map $X\to pt$.
When $X$ is a connected space, there is a canonical isomorphism
$$E_2 ^{*, 0} \cong H^0 (X; R) \otimes H^* (BG; R)\cong H^* (BG; R).$$ 
In this case, the $H^* (BG; R)$-module structure on $E_2^{*,*}$ coincides with the $E_2 ^{*, 0}$-module structure described above.
The $H^* (BG, R)$-module structure is used for calculating differentials of arbitrary elements given the differentials
of the generators of the algebra $H^* (X; R)$; see for instance \cite{carlsson1980}. 
 
We will also use compatibility of two different Serre spectral sequences when 
there is a diagram of fibrations; in particular for the diagram 
$$
\xymatrix{ X \ar[r] & EG \times _G X \ar[r] & BG \\
X  \ar[r] \ar@{=}[u]  & EH \times _H X \ar[r] \ar[u] & BH\ar[u],}
$$
where $H$ is a subgroup of $G$.

Assume that $X$ is a finite-dimensional $G$-CW-complex with cohomology ring
$$H^* (X; \bbF_2) \cong H^* (S^n\times S^m ; \bbF_2) $$ for some $n, m$
satisfying $1\leq n < m$. In fact, with the exception of \cref{prop:generalizations1sn}, for all the results we obtain in this section 
we only use the fact that the mod-$2$ cohomology is isomorphic to $\bbF_2$ in degrees $0$, $n$, $m$, $n+m$ with $1\le n < m$ such that a generator in degree $m+n$ is the product of two generators in degrees $m$ and $n$. So we can also take this weaker cohomology condition as our assumption on $X$.

Since $\bbF_2$ has no nontrivial automorphisms and since $n < m$, the $G$-action on $H^* (X; \bbF_2)$ is trivial. Hence we have 
\[E_2 ^{s,t} \cong H^t (X; \bbF_2) \otimes H^s (BG; \bbF_2)\] with 
$E_2 ^{s,t}$ nonzero only at dimensions $t=0, n, m, n+m$. 
As it is often done, we identify $E_2 ^{0, *} \cong H^* (X; \bbF_2) \otimes H^0 (BG; \bbF_2)$
with $H^* (X; \bbF_2)$ and write elements of $E_2 ^{*, *}$ as a product $tx$ instead of a tensor product 
$t\otimes x$. If $t_1, t_2$ are generators 
for $H^* (X; \bbF_2)$ in dimensions $n$ and $m$, then
$E_2 ^{*,*}$ is a free $H^*(BG; \bbF_2)$-module with generators 
$1$, $t_1$, $t_2$ and $t_1t_2$. 

Since $n<m$, we have $d_i (t_1)=0$ for all $i\leq n$, hence $t_1$ is transgressive.
Let $d_{n+1} (t_1) =\mu_1 \in H^{n+1} (BG; \bbF_2).$ We call the cohomology 
class $\mu_1 \in H^{n+1} (BG; \bbF_2)$ the \emph{$k$-invariant for the 
sphere $S^n$}. If we also have $d_i (t_2)=0$ for $i \leq m$, then $t_2$
is also transgressive and there is a $\mu_2 \in H^{m+1} (BG; \bbF_2)$ such that 
$d_{m+1} (t_2)=\mu_2 \in H^{m+1} (BG; \bbF_2)$ 
modulo the ideal $\langle \mu_1 \rangle$. 
We call a choice of $\mu_2$ in $H^* (BG; \bbF_2)$ 
a \emph{$k$-invariant for the sphere $S^m$}. 

A graded $\bbF_2$-vector space $A^* = \bigoplus_{i\geq 0} A^i $ is called \emph{finite-dimensional} 
if there is a $d$ such that $A^i=0$ for all $i>d$.  
If $P$ is a subgroup of $G$ which acts freely on $X$, then 
$H^*_P (X; \bbF_2)\cong H^* (X/P; \bbF_2)$. Hence $H^i _P (X; \bbF_2)=0$ for $i>\dim X$ and 
$H^* _P (X; \bbF_2)$ is finite-dimensional as a graded $\bbF_2$-vector space.
 
\begin{lemma}\label{lem:mu1} Let $G$ be a finite group, and $X$ 
be a finite-dimensional $G$-CW-complex such that (as $\bbF_2$-algebras)
$$H^* (X; \bbF_2) \cong H^* (S^n\times S^m ; \bbF_2) $$ 
for some $n, m\geq 1$ satisfying $n<m$.
Assume that $G$ has a subgroup $P \cong \mathbb Z/2 \times \mathbb Z/2$ that acts 
freely on $X$. Then the $k$-invariant $\mu_1 \in H^{n+1} (BG; \bbF_2)$ 
for the sphere $S^n$ is not equal to zero.
\end{lemma}

\begin{proof}
 Let $P \cong \bbZ /2\times \bbZ/ 2$ be a subgroup of $G$ that acts freely on $X$. 
Consider the spectral sequence
$$E_2^{i,j} =H^i (BP ; H^j (X; \bbF_2)) \Rightarrow H^*_P (X; \bbF_2)$$ 
for the $P$-action on $X$. Since the $P$-action on $X$ is free, $H^*_P (X ; \bbF_2)$ is finite-dimensional 
as a graded $\bbF_2$-vector space. We denote the differentials of this spectral sequence by $d'_i$.
Assume that $\mu_1=0$. Then by the compatibility of the Serre spectral sequences, we have
$$d'_{n+1} (t_1) =\Res _P ^G  ( d_{n+1} (t_1) )=\Res^G _P \mu _1=0.$$

Note that $H^* (BP ; \bbF_2) \cong \bbF_2 [a, b]$ with $|a|=|b|=1$. In particular, the cohomology
ring $H^* (BP; \bbF_2)$ has no zero divisors.
Assume that $d'_{m-n+1} (t_2)\neq 0$. Then for every nonzero cohomology class $\alpha \in H^* (BP; \bbF_2)$, 
we have $$d' _{m-n+1} (t_2 \alpha) =d' _{m-n+1} (t_2 ) \alpha \neq 0.$$  
Hence the horizontal line $E_2 ^{*, m}$ does not survive to the $E_{m+1}$-page. In this case the
only differentials that can hit the bottom line have to come from the top line. 
But  $E_2 ^{*, 0}$
has Krull dimension $2$, which gives that 
$E_{\infty } ^{*, 0} =E_2 ^{*, 0}/\langle d_{n+m+1} (t_1t_2)\rangle$ has Krull dimension at least $1$.
This is in contradiction with the fact that $E_{\infty} ^{*, 0}$ is finite-dimensional. 
So we must have $d'_{m-n+1}(t_2)=0$. 

Since $d'_{m-n+1} ( t_2)=0$, $t_2$ survives to $E_{n+1}$, and $d'_{n+1} (t_2)=0$ because $E_{n+1}^{*,m-n}=0$ if $m\neq 2n$ and $d'_{n+1}(t_2)=d'_{m-n+1}(t_2)=0$ if $m=2n$.
By the product structure on the spectral sequence, 
this gives that 
$$d'_{n+1} (t_1 t_2 ) =d'_{n+1} (t_1) t_2 + t_1 d'_{n+1} (t_2) =0$$
since both $d'_{n+1} (t_1)$ and $d'_{n+1} (t_2)$ are zero. 
Let  $\mu_2'\coloneqq d'_{m+1} (t_2)$ denote the $k$-invariant for the sphere $S^m$. 
If $\mu_2'=0$, then the horizontal line at $j=m$ 
will survive to the $E_{\infty}$-page which gives a contradiction because 
$H^*_P(X ; \bbF_2)$ is finite-dimensional. Hence $\mu_2'=d'_{m+1}(t_2) \neq 0$. 
This gives that $$d'_{m+1} (t_1t_2)= t_1 \mu_2' \neq 0.$$
Hence the top horizontal line does not survive to the $E_{\infty}$-page, therefore 
there is no other differential hitting the bottom line. We conclude that
$$E^{*, 0} _{\infty} \cong H^* (BP; \bbF_2)/ \langle \mu'_2\rangle.$$ 
This again gives a contradiction because 
the Krull dimension of $H^*(BP; \bbF_2)$ is 2, hence $E^{*, 0} _{\infty}$ is not 
finite-dimensional.
\end{proof}

\begin{lemma}\label{lem:mu2} Let $G$ and $X$ be as in \cref{lem:mu1}. Suppose that $\mu_1$ is a non-zero divisor in $H^* (BG; \bbF_2)$. 
Let $S=H^* (BG, \bbF_2)/\langle \mu _1 \rangle $ denote the quotient ring. 
Then $t_2 \in H^m (X; \bbF_2)$ is transgressive and there is a short exact sequence 
$$ 0 \to (S/\mu_2 S )_i \to H^i_G (X ; \bbF_2) \to (\Ann _S (\mu_2) ) _{i-m} \to 0$$
where the first map is induced by $\pi^*\colon H^* (BG; \bbF_2) \to H^*_G (X; \bbF_2)$.
\end{lemma}

\begin{proof}  
Consider the case where $m\leq 2n$. Write $d_{m-n+1} (t_2)$ in the form $t_1 \gamma$ for some  $\gamma \in H^* (BG; \bbF_2)$. Then $d_{n+1} (t_1 \gamma)=0$ because $t_1\gamma$ is 
in the image of $d_{m-n+1}$. This gives that 
$0=d_{n+1} (t_1 \gamma) = \mu_1\gamma$.
Since $\mu_1$ is a non-zero divisor we get $\gamma=0$, hence $d_{m-n+1} (t_2)=0$.
Now consider the case where $m >2n$. In this case $d_{m-n+1}(t_2)$ lies in the kernel of the map $d_{n+1}\colon E_{n+1} ^{*, n} \to E_{n+1} ^{*, 0}$, which is defined by multiplication with $\mu_1$.
Since $\mu_1$ is a non-zero divisor, we have $\ker d_{n+1}=0$, hence $d_{m-n+1} (t_2)=0$.
Therefore $t_2$ survives to $E_{m+1}$, i.e., it is transgressive in both cases.

By the multiplicative structure of the spectral sequence, we have $d_{n+1} (t_1t_2)= t_2 \mu_1 \neq 0$. 
This gives that $E_{m+1}$ only has two nonzero lines which are at $j=0$ and $j=m$ and both of them are 
isomorphic to $S$. 
The differential between these two nonzero lines is defined by the multiplication 
with $d_{m+1} (t_2) =\mu_2$. Hence we have $E_{\infty} ^{*, 0} \cong S/\mu_2S$ and 
$E_{\infty} ^{*, m} \cong \Ann _S (\mu_2)$. Since the $E_{\infty}$-page has 
only two nonzero horizontal lines at $j=0$ and $j=m$, for every $i \geq 0$, 
there is a short exact sequence $$0 \to E ^{i, 0} _{\infty} \to H^i _G (X; \bbF_2 ) \to E_{\infty} ^{i-m, m} \to 0.$$
Combining with the above observations, this gives the short exact sequence given in the lemma.
\end{proof}

In the rest of the section we assume $G=A_4$. By \cref{thm:AdemMilgram},
the mod-$2$ cohomology algebra of $A_4$ is 
$$H^*(BA_4; \bbF_2) \cong H^*(B(\bbZ/2)^2; \bbF_2)^{C_3}\cong \bbF_2 [u,v,w]/\langle u^3+v^2+vw+w^2\rangle,$$
where $\deg (u)=2$ and $\deg(v)=\deg(w)=3$. From the first isomorphism it is clear
that $H^*(BA_4, \bbF_2)$ has no zero divisors. 
 
\begin{proposition}\label{pro:ForA_4F_2} Let $G=A_4$ and $X$ be a finite, free $G$-CW-complex such that 
$$H^* (X;\bbF_2) \cong H^* (S^n\times S^m ; \bbF_2)$$ for some $n, m\geq 1$ with $n<m$.
Then the following holds:
\begin{enumerate}
\item The homomorphism $$\pi^*\colon H^* (BG; \bbF_2) \to H^* _{G} (X; \bbF_2)$$ is surjective and 
its kernel is the ideal $J\coloneqq\langle \mu_1, \mu_2 \rangle $ generated by the $k$-invariants $\mu_1$ and $\mu_2$.
\item $H^* (BG, \bbF_2)/J$ is finite over $\bbF_2$.
\item The ideal $J$ is closed under Steenrod operations.
\end{enumerate}
\end{proposition}

\begin{proof}   
Consider the Serre spectral sequence in mod-$2$ coefficients. By \cref{lem:mu1}, the $k$-invariant $\mu_1$ is nonzero.
The ring $H^*(BG; \bbF_2)$ has no zero-divisors, hence we can apply \cref{lem:mu2} to conclude that 
the generator $t_2$ is transgressive.
So the $k$-invariant $\mu_2$ is defined and as in the proof of \cref{lem:mu2}, the $E_{m+1}$-page 
has only two nonzero lines at $j=0$ and $j=m$.
Since $S/\mu_2S= H^* (BG; R) /J$ where $J=\langle \mu_1, \mu_2 \rangle$, 
we conclude that
the kernel of $\pi^*$ is the ideal $J$. The cohomology ring $H_G ^* (X; \bbF_2)$ is
finite-dimensional. This gives that $H^* (BG; \bbF_2)/J$ is finite-dimensional, hence finite. 
The homomorphism $\pi^*$ is induced by a continuous map between 
topological spaces, thus its kernel $J$ is closed under Steenrod operations. 

Since $H^* (BG; \bbF_2)/J$ 
is finite, we conclude that $\mu_1, \mu_2$ is a regular sequence by \cref{lem:Coprime}. Thus $\mu_2$ is a non-zero divisor in $S$.
This gives that $ (\Ann _S (\mu_2) ) _{i-m}=0$ for $S=H^* (BG; \bbF_2)/\langle \mu_1\rangle$. 
Thus $E^{i,j} _{\infty}=0$ for all $j>0$. Hence $\pi^*$ is surjective.
\end{proof}

\begin{remark} \cref{pro:ForA_4F_2} and the following theorem hold more generally for $X$ with four-dimensional cohomology; see \cite[Remark~4.6]{ruepingstephan2024}. Moreover, the assertion from \eqref{lem:mu1} and \eqref{lem:mu2} that $J$ is a parameter ideal follows from the algebraic result \cite[Proposition~4.3]{ruepingstephan2024}.
\end{remark}

We obtain the following:

\begin{theorem}\label{thm:ObstructionskInvariants}  
Let $G=A_4$ and $X$ be a finite, free $G$-CW-complex such that
$$H^* (X;\bbF_2) \cong H^* (S^n\times S^m ; \bbF_2)$$ for some $n, m\geq 1$ with $n<m$.
The ideal $J$ generated by
the $k$-invariants is equal to one of the ideals listed in \cref{thm:Intro-classification_steenrodclosedparameterideals}.    
\end{theorem}

\begin{proof}
By \cref{pro:ForA_4F_2}, the ideal $J$ is a parameter ideal closed 
under Steenrod operations. Hence the result follows from \cref{thm:Intro-classification_steenrodclosedparameterideals}.
\end{proof}

When the $G$-CW-complex $X$ has the integral cohomology of a product of two spheres, the 
possibilities for the ideal $J$ are more restricted. To see this we use the commutativity of Steenrod
operations with transgressions.
 
 \begin{proposition}\label{pro:Steenrod} Let $G=A_4$ and $X$ be a finite, free $G$-CW-complex
such that $$H^* (X;\bbF_2) \cong H^* (S^n\times S^m ; \bbF_2)$$ for some $n, m\geq 1$ with $n<m$.
Let $t_1$ and $t_2$ denote the generators of
$H^* (X ; \bbF_2)$ in dimensions $n$ and $m$, and let $\mu_1$ and 
$\mu_2$ in $H^*(BA_4; \bbF_2)$ be the $k$-invariants for the $A_4$-action on $X$.  
If $\Sq^{m-n} (t_1)=t_2 \lambda$, then  $\Sq^{m-n} (\mu_1)=\mu_2 \lambda$ modulo $\langle \mu_1 \rangle$.
\end{proposition}

\begin{proof} 
By \cite[Corollary~6.9]{mccleary2001user}, the transgressions in the Serre spectral sequence  
commute with the Steenrod operations. This means that if $x \in H^n (X ; \bbF_2)$ 
survives to the $E_{n+1}$-page, then for every $i\geq 0$, $\Sq^i (x)$ survives to the $E_{n+i+1}$-page
and the equality $$d_{n+i+1} ( \Sq^i (x))=\Sq^i (d_{n+1} (x))$$ holds on the $E_{n+i+1}$-page. 
Using this, we can conclude that if $\Sq^{m-n} (t_1)=t_2 \lambda$, then  
$$
\Sq^{m-n} (\mu_1) =\Sq ^{m-n} ( d_{n+1} (t_1))= d_{m+1} ( \Sq ^{m-n} (t_1) )= d_{m+1} ( t_2 \lambda) \\
=\mu_2 \lambda$$
modulo $\langle \mu_1 \rangle$. This completes the proof. 
\end{proof} 

As a consequence we obtain the following theorem. 
 
\begin{theorem}\label{thm:ObstructionskInvariants_integral} 
Let $G=A_4$ and $X$ be a finite, free $G$-CW-complex such that
$$H^* (X;\bbZ) \cong H^* (S^n\times S^m ; \bbZ)$$ for some $n, m\geq 1$ with $n<m$.
Let $J$ be the ideal generated by
the $k$-invariants such that
$H^* (X/G ; \bbF_2) \cong H^* (BA_4; \bbF_2) /J$.
Then the ideal $J$ is equal to one of the ideals listed in \cref{thm:Intro-classification_steenrodclosedparameterideals}
except the twisted ones listed in \eqref{item:twisted}.
\end{theorem}
 
\begin{proof} By the universal coefficient theorem for cohomology, the cohomology ring $H^* (X; \bbF_2)$
satisfies the conditions of \cref{thm:ObstructionskInvariants}, so we have   $H^* (X/G ; \bbF_2) \cong H^* (BG; \bbF_2) /J$,
where the ideal $J$ must be equal to one of the ideals listed in \cref{thm:Intro-classification_steenrodclosedparameterideals}.  
Consider the homomorphism  $m_2^*\colon H^* (X; \bbZ) \to H^* (X; \bbF_2)$ induced by the mod-$2$ reduction map $m_2\colon \bbZ \to \bbF_2$. 
Since the generator $t_1$ is in the 
image of the mod-$2$ reduction map, $\beta_0 (t_1)=0$ where $\beta_0$ is the connecting homomorphism 
for $0 \to \mathbb Z \to \mathbb Z \to \bbF_2\to 0$. The Bockstein homomorphism is $\beta =m_2^* \circ \beta_0$, 
hence we have $\beta (t_1)=0$.
Since $\beta (t_1)=\Sq^1 (t_1)$, we can apply \cref{pro:Steenrod} and conclude that $\Sq ^1( \mu_1)=0$ mod $\langle \mu_1 \rangle$.
This shows that the ideal $J=\langle \mu_1, \mu _2 \rangle$ can not be twisted.
\end{proof} 

Removing the degrees of the parameters for the twisted ideals from \cref{cor:degreesofparameters} yields the list in the following result.
\begin{corollary}\label{cor:DimensionsFreeAction}
Let $G=A_4$ and $X$ be a finite, free $G$-CW-complex such that $$H^* (X;\bbZ) \cong H^* (S^n\times S^m ; \bbZ)$$ for some $n, m\geq 1$.

Then the unordered pair $(n+1,m+1)$ must be one of the following: 
\begin{enumerate}
\item $(3k,2l)$ where $k$ is not larger than the highest power of $2$ dividing $l$, or
\item $(3i+2^{s+r+1}-2^{s+1},2^{s+r+1}-2^{s})$ for $s\geq 1$, $r\geq 1$ and $0\leq i <2^{s-1}$.
\end{enumerate}
\end{corollary}

It follows from \cref{lem:rlogr} that the number of pairs $(n,m)$ (as above) with $n,m\le r$ is in $O(r\log r)$. Since the number of pairs $(n,m)$ with $n,m\leq r$ grows like $r^2$, we obtain the following result:

\begin{corollary}\label{cor:percentage} For $r\geq 1$, the percentage of those pairs $(n,m)$ with $n,m\leq r$ such that there exists a finite, free $A_4$-CW-complex $X\simeq S^n\times S^m$ tends to zero as $r\to \infty$.
\end{corollary}

Another consequence of \cref{pro:Steenrod} is a slight generalization of a result by Blaszczyk \cite[Proposition~4.3]{blaszczyk2013free}.

\begin{proposition}\label{prop:generalizations1sn} The group $G=A_4$ can not act freely on a finite CW-complex $X$
with $H^* (X ; \bbF_2) \cong H^* (S^1 \times S^m ; \bbF_2)$ for any $m\geq 1$.   
\end{proposition}

\begin{proof} The case $m=1$ holds by \cref{thm:nofreeactionequidim}. If $m\geq 2$, then the $k$-invariant $\mu_1$ must be equal to $u\in H^2 (BA_4; \bbF_2)$. Since $\Sq^1(t_1)=t_1^2=0$, we must have $\Sq^1(u)=0$ modulo the ideal $\langle u \rangle$. However $\Sq^1(u)=v$ is not in the ideal $\langle u \rangle$.
\end{proof}
 
Extending the above arguments we can say that a parameter ideal $\langle \mu_1, \mu_2 \rangle$ 
with $\Sq^{m-n}(\mu_1)=\mu_2$ can not be realized by a finite, free $A_4$-CW-complex $X$ homotopy equivalent to $S^n \times S^m$.

\begin{corollary}\label{cor:NonRealizablePairs}
Let $G=A_4$ and $\langle x, Sq^1(x) \rangle$ be a twisted ideal. For each $r\geq 0$, 
the ideals of the form $\langle x^{2^r} ,\Sq^1(x)^{2^r}\rangle$  
are not realizable by a finite, free $G$-CW-complex $X$ homotopy equivalent to $S^n \times S^m$.
\end{corollary}

\begin{proof} The statement follows from \cref{pro:Steenrod} since
$\Sq^{2^r} (x^{2^r} )= \Sq^1(x)^{2^r}$.
\end{proof}
In particular by \cref{cor:NonRealizablePairs}, the fibered pairs $(v^{2^r},u^{2^r})$ are not realizable 
by a finite, free $A_4$-CW-complex $X$ homotopy equivalent to $S^n \times S^m$.

\newpage
\appendix
\section{Table of Steenrod closed parameter ideals}\label{sec:appendix}
The following table visualizes the degrees of parameters for fibered, twisted and mixed Steenrod closed parameter ideals.
\tikzstyle{twisted}=[diamond,draw,minimum size=10,scale=0.6]
\tikzstyle{mixed}=[circle,draw,minimum size=5,scale=0.8]
\tikzstyle{fibered}=[rectangle,draw,minimum size=5]
\tikzstyle{notrealizable}=[fill=gray!255]
\tikzstyle{realizable}=[fill=gray!10]
\tikzstyle{realizableunclear}=[fill=gray!48]
\begin{figure}[H]
\centering
\begin{tikzpicture}
\newcommand{\drawdot}[3]
  {\node[#1,scale=0.6]  (#1#2#3) at ({max(.2*#2,.2*#3)},{min(.2*#2,.2*#3)}) {};}
\newcommand{\drawcoord}[1]
  {\node  (coord#1) at (0,.2*#1) {#1};
  \node  (coord#1) at (.2*#1,0) {#1};
  \draw [dotted] (.2,.2*#1) -- (.2*60,.2*#1);
  \draw [dotted] (.2*#1,.2) -- (.2*#1,60*.2);}
\newcommand{\fibered}[2]{\drawdot{fibered}{2*#1}{3*#2}}
\newcommand{\twisted}[1]{\drawdot{twisted}{2^#1-1}{2^#1-2}}
\newcommand{\mixed}[3]{\drawdot{mixed}{3*#1+(2^#3-2)*2^#2}{(2^#3-1)*2^#2}}
\drawdot{fibered,realizable}{2}{3}
\drawdot{fibered,realizable}{4}{3}
\drawdot{fibered,realizable}{4}{6}
\drawdot{fibered,realizable}{6}{3}
\drawdot{fibered,realizable}{8}{3}
\drawdot{fibered,realizable}{8}{6}
\drawdot{fibered,realizable}{8}{9}
\drawdot{fibered,realizable}{8}{12}
\drawdot{fibered,realizable}{10}{3}
\drawdot{fibered,realizable}{12}{3}
\drawdot{fibered,realizable}{12}{6}
\drawdot{fibered,realizable}{14}{3}
\drawdot{fibered,realizable}{16}{3}
\drawdot{fibered,realizable}{16}{6}
\drawdot{fibered,realizable}{16}{9}
\drawdot{fibered,realizable}{16}{12}
\drawdot{fibered,realizable}{16}{15}
\drawdot{fibered,realizable}{16}{18}
\drawdot{fibered,realizable}{16}{21}
\drawdot{fibered,realizable}{16}{24}
\drawdot{fibered,realizable}{18}{3}
\drawdot{fibered,realizable}{20}{3}
\drawdot{fibered,realizable}{20}{6}
\drawdot{fibered,realizable}{22}{3}
\drawdot{fibered,realizable}{24}{3}
\drawdot{fibered,realizable}{24}{6}
\drawdot{fibered,realizable}{24}{9}
\drawdot{fibered,realizable}{24}{12}
\drawdot{fibered,realizable}{26}{3}
\drawdot{fibered,realizable}{28}{3}
\drawdot{fibered,realizable}{28}{6}
\drawdot{fibered,realizable}{30}{3}
\drawdot{fibered,realizable}{32}{3}
\drawdot{fibered,realizable}{32}{6}
\drawdot{fibered,realizable}{32}{9}
\drawdot{fibered,realizable}{32}{12}
\drawdot{fibered,realizable}{32}{15}
\drawdot{fibered,realizable}{32}{18}
\drawdot{fibered,realizable}{32}{21}
\drawdot{fibered,realizable}{32}{24}
\drawdot{fibered,realizable}{32}{27}
\drawdot{fibered,realizable}{32}{30}
\drawdot{fibered,realizable}{32}{33}
\drawdot{fibered,realizable}{32}{36}
\drawdot{fibered,realizable}{32}{39}
\drawdot{fibered,realizable}{32}{42}
\drawdot{fibered,realizable}{32}{45}
\drawdot{fibered,realizable}{32}{48}
\drawdot{fibered,realizable}{34}{3}
\drawdot{fibered,realizable}{36}{3}
\drawdot{fibered,realizable}{36}{6}
\drawdot{fibered,realizable}{38}{3}
\drawdot{fibered,realizable}{40}{3}
\drawdot{fibered,realizable}{40}{6}
\drawdot{fibered,realizable}{40}{9}
\drawdot{fibered,realizable}{40}{12}
\drawdot{fibered,realizable}{42}{3}
\drawdot{fibered,realizable}{44}{3}
\drawdot{fibered,realizable}{44}{6}
\drawdot{fibered,realizable}{46}{3}
\drawdot{fibered,realizable}{48}{3}
\drawdot{fibered,realizable}{48}{6}
\drawdot{fibered,realizable}{48}{9}
\drawdot{fibered,realizable}{48}{12}
\drawdot{fibered,realizable}{48}{15}
\drawdot{fibered,realizable}{48}{18}
\drawdot{fibered,realizable}{48}{21}
\drawdot{fibered,realizable}{48}{24}
\drawdot{fibered,realizable}{50}{3}
\drawdot{fibered,realizable}{52}{3}
\drawdot{fibered,realizable}{52}{6}
\drawdot{fibered,realizable}{54}{3}
\drawdot{fibered,realizable}{56}{3}
\drawdot{fibered,realizable}{56}{6}
\drawdot{fibered,realizable}{56}{9}
\drawdot{fibered,realizable}{56}{12}
\drawdot{fibered,realizable}{58}{3}
\drawdot{fibered,realizable}{60}{3}
\drawdot{fibered,realizable}{60}{6}
\drawdot{twisted,realizable}{2}{3}
\drawdot{twisted,realizable}{6}{7}
\drawdot{twisted,realizable}{14}{15}
\drawdot{twisted,realizable}{30}{31}
\drawdot{mixed,realizable}{11}{12}
\drawdot{mixed,realizable}{19}{24}
\drawdot{mixed,realizable}{22}{24}
\drawdot{mixed,realizable}{25}{24}
\drawdot{mixed,realizable}{35}{48}
\drawdot{mixed,realizable}{38}{48}
\drawdot{mixed,realizable}{41}{48}
\drawdot{mixed,realizable}{44}{48}
\drawdot{mixed,realizable}{47}{48}
\drawdot{mixed,realizable}{50}{48}
\drawdot{mixed,realizable}{53}{48}
\drawdot{mixed,realizable}{12}{14}
\drawdot{mixed,realizable}{24}{28}
\drawdot{mixed,realizable}{27}{28}
\drawdot{mixed,realizable}{48}{56}
\drawdot{mixed,realizable}{51}{56}
\drawdot{mixed,realizable}{54}{56}
\drawdot{mixed,realizable}{57}{56}
\drawdot{mixed,realizable}{28}{30}
\drawdot{mixed,realizable}{56}{60}
\drawdot{mixed,realizable}{59}{60}
\node  (coord0) at (0,0) {0};
\drawcoord{5}
\drawcoord{10}
\drawcoord{15}
\drawcoord{20}
\drawcoord{25}
\drawcoord{30}
\drawcoord{35}
\drawcoord{40}
\drawcoord{45}
\drawcoord{50}
\drawcoord{55}
\drawcoord{60}
\matrix[draw,below left,fill=white,shift={(-3,-0.3)}] at (current bounding box.north) {
  \node [fibered,label={right:fibered}] {}; \\
  \node [twisted,label={right:twisted}] {}; \\
  \node [mixed,label={right:mixed}] {}; \\
};
\end{tikzpicture}
\caption{Steenrod closed parameter ideals with parameters of degrees at most $60$.} \label{fig:firstIdeals2}
\end{figure}
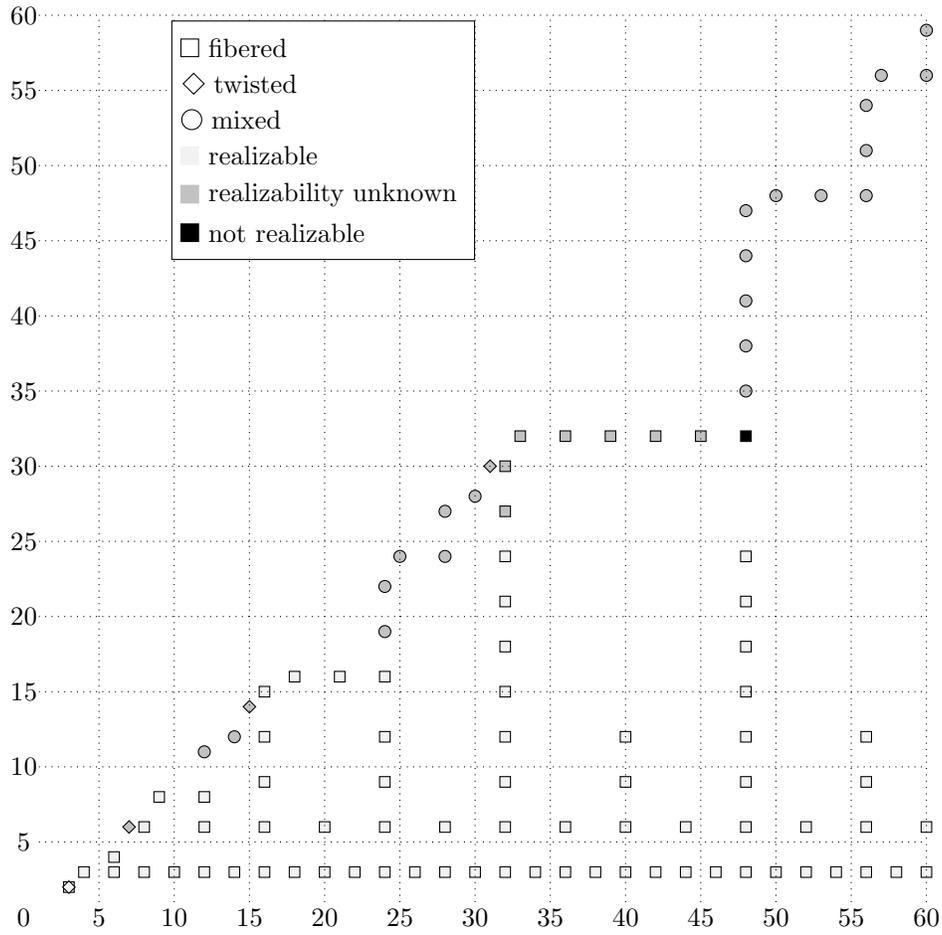

\providecommand{\bysame}{\leavevmode\hbox to3em{\hrulefill}\thinspace}
\providecommand{\MR}{\relax\ifhmode\unskip\space\fi MR }
\providecommand{\MRhref}[2]{%
  \href{http://www.ams.org/mathscinet-getitem?mr=#1}{#2}
}

\end{document}